\documentclass[a4paper,10pt]{amsart}
\usepackage{eurosym}
\usepackage{amsmath}
\usepackage{mathrsfs}
\usepackage{amsfonts}
\usepackage{graphicx}
\usepackage{color}
\usepackage{amsfonts}
\usepackage{amssymb}
\usepackage{amsthm}
\usepackage[hidelinks]{hyperref}
\usepackage[normalem]{ulem}
\usepackage{esint}

\usepackage{palatino}

\usepackage[abbrev,backrefs]{amsrefs}
\usepackage[hidelinks]{hyperref}

\usepackage{mathtools}
\usepackage{tikz-cd}
\usepackage{float}

\newtheorem{theorem}{Theorem}[section]

\newtheorem{openquestion}[theorem]{Open Question}
\newtheorem{corollary}[theorem]{Corollary}

\newtheorem{lemma}[theorem]{Lemma}

\newtheorem{remark}[theorem]{Remark}

\numberwithin{equation}{section}

\def\N{{\mathbb N}}

\def\I{{\mathcal I}}

\newcommand{\R}{\mathbb{R}}
\newcommand{\rn}{\mathbb{R}^n}

\newcommand{\dx}{\,{\rm d}x}
\newcommand{\dy}{\,{\rm d}y}
\newcommand{\dz}{\,{\rm d}z}
\newcommand{\dn}{\,{\rm d}\nu}
\newcommand{\ds}{\,{\rm d}s}
\newcommand{\du}{\,{\rm d}u}
\newcommand{\dt}{\,{\rm d}t}

\newcommand{\barint}{
	\rule[.036in]{.12in}{.009in}\kern-.16in \displaystyle\int }

\newcommand{\barcal}{\mbox{$ \rule[.036in]{.11in}{.007in}\kern-.128in\int $}}

\usepackage{enumitem}
\makeatletter
\let\@wraptoccontribs\wraptoccontribs
\makeatother

\mathchardef\mhyphen="2D

\title{On sharp constants in higher order Adams-Cianchi inequalities}

\author[P. Roychowdhury]{Prasun Roychowdhury}
\address[P. Roychowdhury]{Mathematics Division, National Center for Theoretical Sciences, NTU, Cosmology Building, No. 1, Sec. 4, Roosevelt Rd., Taipei City, Taiwan 106, R.O.C.
}
\email{prasunroychowdhury1994@gmail.com}

\author[D. Spector]{Daniel Spector}
\address[D. Spector]{Department of Mathematics, National Taiwan Normal University, No. 88, Section 4, Tingzhou Road, Wenshan District, Taipei City, Taiwan 116, R.O.C.\\
\newline
and
\newline
National Center for Theoretical Sciences\\No. 1 Sec. 4 Roosevelt Rd., National Taiwan
University\\Taipei, 106, Taiwan
\newline
and
\newline
Department of Mathematics, University of Pittsburgh, Pittsburgh, PA 15261 USA
}
\email{spectda@protonmail.com}

\date{\today}

\begin{document}
	
	\maketitle
	
	\begin{abstract}
The main results of this paper are the establishment of sharp constants for several families of critical Sobolev embeddings.  These inequalities were pioneered by David R. Adams, while the sharp constant in the first order case is due to Andrea Cianchi.  We also prove a trace improvement of an inequality obtained independently by K. Hansson and H. Brezis and S. Wainger.
	\end{abstract}

	\section{Introduction}
	Let $\Omega \subset \mathbb{R}^n$ be an open, bounded set and suppose $u \in W^{k,p}_0(\Omega)$ with $kp=n$.  This regime is critical for the Sobolev embedding, where a well-known result is the exponential integrability of $u$ at an appropriate power:  There exist constants $c,C>0$ such that
	\begin{align}\label{critical}
		\int_\Omega \exp\left(c\left|\frac{u(x)}{\| \nabla^k u\|_{L^p}}\right|^{p'}\right) \;dx \leq C,
	\end{align}
	for all $u \in W^{k,p}_0(\Omega)$, where $\nabla^k u$ is the usual iterated gradient, i.e. $\nabla^k u = \nabla (\nabla^{k-1}u)$, and $\| \nabla^k u\|_{L^p} \equiv \| |\nabla^k u| \|_{L^p(\Omega)}$.  The history of such an estimate has been accounted for many times (see, e.g. \cite{CP}), in particular the assertions of Yudovich concerning estimates for Riesz potentials \cite[Theorem 3 on p.~808]{Y:1962} and of Pohozaev in two dimensions  \cite[Lemma 1 on p.~1409]{P:1965}, Peetre's result for the Besov spaces \cite[Theorem 8.2 on p.~302]{P:1966}, Trudinger's result with the best exponent in the first order case \cite[Theorem 2 on p.~478]{T:1967}, Strichartz's result for Bessel potentials which obtains the correct exponent \cite{S:1971}, and Hedberg's result for Riesz potentials \cite{H:1972}.  In particular, the inequality \eqref{critical} can be argued by Hedberg's estimate and a simple pointwise bound for a function in terms of its Riesz potential:
	\begin{align}\label{representation_inequality}
		|u(x)| \leq c I_k|\nabla^k u|(x)
	\end{align}
	for some $c=c(k,n)>0$ (see Section \ref{preliminaries} for a precise definition of $I_\alpha$, $\alpha \in (0,n)$).
	
	It seems less well-known that the inequality \eqref{critical} admits an improvement with respect to the dimension of integration, the trace exponential integrability inequality:  For every $d \in (0,n]$, there exist constants $c_d,C_d>0$
	\begin{align}\label{critical_trace}
		\int_\Omega \exp\left(c_d \left|\frac{u(x)}{\| \nabla^k u\|_{L^p}}\right|^{p'}\right) \;d\nu \leq C_d
	\end{align}
	for all $u \in W^{k,p}_0(\Omega)$ and for all Radon measures $\nu$ which satisfy the ball growth condition 
	\begin{align}\label{bgc}
		\nu(B(x,r))\leq C_d' r^d
	\end{align}
for all $x \in \mathbb{R}^n$ and $r>0$.
Indeed, when $d =k \in \mathbb{N}$ and $\nu=\mathcal{H}^k|_{\Omega \cap H}$ for any hyperplane $H$, an inequality in the spirit of \eqref{critical_trace} for Riesz potentials is a more precise recounting of Yudovich's assertion in \cite[Theorem 3 on p.~808]{Y:1962}.  For general Radon measures $\nu$ which satisfy \eqref{bgc}, the analogous estimate for a Bessel potential is due to D.R. Adams \cite[Theorem 3(ii) on p.~912]{Adams1973}, while an estimate involving Riesz potentials can be found in the paper of the second author and \'{A}. Martinez \cite{MS} (see also \cite{AdamsXiao, AdamsXiao2, Chen-Spector} for similar critical estimates and \cite{RSS} for a related dual estimate).  As in the estimate \eqref{critical}, the inequality \eqref{critical_trace} can be argued by a Riesz potential estimate \cite[Corollary 1.4 on p.~880]{MS} and the representation inequality \eqref{representation_inequality}.
	
	Concerning the family of inequalities \eqref{critical}, there is a fairly complete picture with regard to the best constant.  In the first order case, Moser \cite{Moser} showed that one may take $c=n\omega_n^{n/(n-1)}$  in the inequality \eqref{critical} and that the inequality fails for any larger value of $c$.  This first order case with the best constant has been called the Moser-Trudinger inequality, while when the order of the derivative is greater than one, D.R. Adams gave a fundamental contribution toward a higher order Moser-Trudinger inequality in his paper \cite{Adams1988}, which paved the way for the establishment of the best constant in this setting.  In particular, in his paper, Adams obtains the best constant for the case of the analogous Riesz potential embedding, which gives the sharp constant for $p=\frac{n}{k}=2$, while the remaining cases were subsequently settled in the paper \cite{SS}.  The program initiated by Adams and completed in \cite{SS} involves two main steps:  
	\begin{enumerate}
		\item Obtain the best constant in the analogous Riesz potential embedding; 
		\item Derive and utilize a sharp form of the representation formula \eqref{representation_inequality}.
	\end{enumerate}
	The second step played a more minor role in Adams' paper, as in place of $\nabla^k u$, he utilized the differential operator
	\begin{align*}
		D^k u :=\left\{
		\begin{array}{ll}
			\nabla (-\Delta)^{(k-1)/2} u , &\quad \text{k odd},\\
			(-\Delta)^{k/2} u , &\quad \text{k even}.\\
		\end{array} 
		\right.
	\end{align*}
	This simplifies the sharp representation formulas bounding $u$ in terms of the Riesz potential of $|D^k u|$, as a similar inequality between $u$ and the Riesz potential of $|\nabla^k u|$ turns out to be more subtle.  This sharp representation was obtained by I. Shafrir and the second author in \cite{SS}, based on Shafrir's work on a sharp constant in the critical $L^\infty$ embedding \cite{Shafrir} and ideas developed by the second author and R. Garg \cite{GS,GS1}.

	The purpose of this paper is to give an analogous treatment of the question of best constant for the family of inequalities \eqref{critical_trace}.  To this end, we first recall the remarkable result of A. Cianchi, who among other results, in \cite{Cianchi:2008} obtained the sharp constant in the inequality \eqref{critical_trace} in the first order case, as well as when $|\nabla u|$ is in a critical Lorentz space.  Cianchi's result is a natural extension of Moser's celebrated result \cite{Moser} concerning the sharp constant for first-order gradients in \eqref{critical}, and we will refer to this first order case as the Adams-Cianchi inequality.   Concerning the higher order Adams-Cianchi inequality, a fundamental contribution was made in the paper of Fontana and Morpurgo \cite{FontanaMorpurgo}, where they obtained the sharp constant in the analogue of Adams' result \cite{Adams1988} for the Riesz potential where the Lebesgue measure is replaced by general Radon measures $\nu$ which satisfy \eqref{bgc}, along with the sharp constant for the analogous inequality utilizing Adams' derivative $D^ku$.  As in the case of Adams' work, this includes the sharp constant in \eqref{critical_trace} for $p=\frac{n}{k}=2$ (see \cite[Theorem 12 on p.~5104]{FontanaMorpurgo}).  That one cannot improve the constant requires a non-degeneracy condition on the measure, that the general upper bound has a matching lower bound at some place in the domain:  There exists $x_0 \in \Omega$, $ \varrho_0>0$, and $C_0>0$ such that
	\begin{align}\label{condition-1}
		C_0 \varrho^d\leq \nu(B(x_0,\varrho)\cap \Omega) \text{ for all } 0<\varrho\leq \varrho_0.
	\end{align}

	Thus one observes that a principle remaining question is that of the establishment of sharp constants in the family of inequalities \eqref{critical_trace} for $k \in \mathbb{N}, 1<k<n$ for $p \neq 2$.  Secondary points of interest (the latter two anticipated by Fontana and Morpurgo \cite[Remark 1 on p.~5073 and the bottom of p.~5093]{FontanaMorpurgo}) include whether one has a similar sharp trace inequalities for
	\begin{enumerate}
		\item  $\nabla^k u$ in critical Lorentz spaces;
		\item the Riesz potentials acting on functions in critical Lorentz spaces;
		\item  Adams' derivative $D^k u$ in critical Lebesgue and Lorentz spaces.
	\end{enumerate}
Note that for the Lebesgue measure, the Lorentz scale extension of Adams' result in $(2)$ and $(3)$ have been addressed by A. Alvino, V. Ferone, G. Trombetti \cite{AFT} and A. Alberico \cite{Alberico}.  

In this paper, we obtain the sharp constants in the family of inequalities \eqref{critical_trace} and complete the picture concerning $(1),(2),$ and $(3)$ for the Lorentz spaces $L^{p,q}(\Omega)$ for $p$ critical and $1<q<+\infty$.  In particular, we demonstrate the validity of these various inequalities with a particular constant, and assuming the non-degeneracy condition \eqref{condition-1}, show the optimality of these constants.     The computations for optimality involve variants of a standard test function for this regime -- the logarithm of the distance to a point.  For $\nabla^k, D^k$ we utilize the test functions from \cite{Shafrir, SS}, while for the Riesz potential we use Adams' capacitary test functions \cite{Adams1988}.  

Our first result is an answer to the principle question, on the sharp constants in the higher order Adams-Cianchi inequalities \eqref{critical_trace}, including their extension to the Lorentz scale, the following 
	\begin{theorem}\label{main-thm-0}
		Let $\Omega$ be an open bounded subset of $\rn$, $n\geq 2$. Let $k \in \mathbb{N}$, $1\leq k<n$, $1<q<\infty$, and set $q'=q/(q-1)$.  Let $\nu$ be any positive Borel measure on $\Omega$ which satisfies \eqref{bgc} and \eqref{condition-1} for some $d\in (0,n]$.  Then there exists positive constant $C=C(n,k,\Omega, C'_d, q)$ such that 
		\begin{align}\label{eqn-ss-lorentz}
			\int_\Omega \exp \left( \beta |u(x)|^{q^\prime}\right) \dn \leq C,
		\end{align}
		for every $\beta \leq	d n^{\frac{q'}{q}}\,\omega_n^{\frac{k}{n}q'}\,\sqrt{\ell^k_n}^{q'}$ and every $u\in W_0^k L^{\frac{n}{k},q}(\Omega)$ with $\|\, |\nabla^k u|\, \|_{L^{\frac{n}{k},q}(\Omega)}\leq 1$.  Moreover, the result is sharp in the sense that the l.h.s.~of \eqref{eqn-ss-lorentz}, with any $\beta>d n^{\frac{q'}{q}}\,\omega_n^{\frac{k}{n}q'}\,\sqrt{\ell^k_n}^{q'}$, cannot be uniformly bounded for every $u\in W_0^k L^{\frac{n}{k},q}(\Omega)$ with \sloppy{
$\|\, |\nabla^k u|\, \|_{L^{\frac{n}{k},q}(\Omega)}\leq 1$}.
	\end{theorem}
	
	Here $\sqrt{\ell^k_n}$ is a combinatorial constant such that
	\begin{align}\label{combinatorial_constant}
		|\nabla^k \log |x| | = \frac{\sqrt{\ell^k_n}}{|x|^k},
	\end{align}
for $x \neq 0$, see Section \ref{preliminaries} below.  The appearance of this constant in optimality emerged through a technique pioneered by I. Shafrir in \cite{Shafrir} and further developed by Shafrir and the second author in \cite{SS}:  Up to proper normalization, $\log |x|$ is the fundamental solution of the family of elliptic equations
	\begin{align*}
		(-1)^k \operatorname*{div_k} |x|^{2k-n} \nabla^k u = 0.
	\end{align*}
	This fact yields \eqref{representation_inequality} with a constant that finds no loss in the critical regime of interest to us.  This reduces the question to a sharp estimate for Riesz potentials, and so the proof of Theorem \ref{main-thm-0} follows from this representation and
	\begin{theorem}\label{main-thm}
Let $\Omega$ be an open bounded subset of $\rn$, $n\geq 2$. Let $\alpha \in (0,n)$, $1<q<\infty$, and set $q'=q/(q-1)$.  Let $\nu$ be any positive Borel measure on $\Omega$ which satisfies \eqref{bgc} and \eqref{condition-1} for some $d\in (0,n]$.  Then there exists positive constant $C=C(n,\alpha,\Omega, C'_d, q)$ such that 
		\begin{align}\label{eqn-main-thm}
			\int_\Omega \exp \left(\beta |I_\alpha f(x)|^{q'} \right) \dn \leq C,
		\end{align}
for every $\beta \leq \frac{d}{n}\gamma(\alpha)^{q'}\omega_n^{-\frac{n-\alpha}{n}q'}$ and every $f\in L^{\frac{n}{\alpha},q}(\Omega)$ with $\|\, |f|\, \|_{L^{\frac{n}{\alpha},q}(\Omega)}\leq 1$.  		
Moreover, the result is sharp in the sense that the l.h.s.~of \eqref{eqn-main-thm}, with any $\beta>\frac{d}{n}\gamma(\alpha)^{q'}\omega_n^{-\frac{n-\alpha}{n}q'}$, cannot be uniformly bounded for every $f\in L^{\frac{n}{\alpha},q}(\Omega)$ with $\|\, |f|\, \|_{L^{\frac{n}{\alpha},q}(\Omega)}\leq 1$.
	\end{theorem}
	Note that the preceding theorem also gives an answer to $(2)$.  Concerning $(3)$, we prove the following theorem which gives the sharp constant for the differential operators employed by Adams \cite{Adams1988} and Alberico \cite{Alberico}.
	\begin{theorem}\label{main-thm-1}
		Let $\Omega$ be an open bounded subset of $\rn$, $n\geq 2$. Let $k \in \mathbb{N}$, $1\leq k<n$, $1<q<\infty$, and set $q'=q/(q-1)$.  Let $\nu$ be any positive Borel measure on $\Omega$ which satisfies \eqref{bgc} and \eqref{condition-1} for some $d\in (0,n]$.  Then there exists positive constant $C=C(n,k,\Omega, C'_d, q)$ such that 
\begin{align}\label{eqn-adams-lorentz}
			\int_\Omega \exp \left( \beta |u(x)|^{q^\prime}\right) \dn \leq C,
\end{align}
for every $\beta \leq \beta_{n,k,q}$ (as defined in \eqref{beta-n-k}) and every $u\in W_0^k L^{\frac{n}{k},q}(\Omega)$ with $\|\, |D^k u|\, \|^q_{L^{\frac{n}{k},q}(\Omega)} \leq 1$. Moreover, the result is sharp in the sense that the l.h.s.~of \eqref{eqn-adams-lorentz}, with $\beta>\beta_{n,k,q}$, cannot be uniformly bounded for every $u\in W_0^k L^{\frac{n}{k},q}(\Omega)$ with \sloppy{$\|\, |D^k u|\, \|_{L^{\frac{n}{k},q}(\Omega)}\leq 1$}.
	\end{theorem}
	
	As asserted in Theorems \ref{main-thm-0},  \ref{main-thm}, and  \ref{main-thm-1}, the constants are sharp in the sense that the coefficient multiplying $u$ may not be replaced by any larger number.  We next turn our attention to the endpoint $q=+\infty$.  Here, as in Cianchi's work, we establish the validity of the inequality on an open interval.  	For the higher order case involving the Riesz potential, we are able to prove the complete analogue, as we demonstrate a failure at the endpoint via capacitary test functions.  This gives a complete picture in the following analogue of Theorem \ref{main-thm}.	
	\begin{theorem}\label{main-thm-infty}
	Let $\Omega$ be an open bounded subset of $\rn$, $n\geq 2$. Let $\alpha \in (0,n)$.  Let $\nu$ be any positive Borel measure on $\Omega$ which satisfies \eqref{bgc} and \eqref{condition-1} for some $d\in (0,n]$.  Let $\gamma<\left(\frac{d}{n}\right)\gamma(\alpha)\,\omega_n^{-\frac{n-\alpha}{n}}$.  Then there exists positive constant $C=C(n,\alpha,\Omega, C'_d,\gamma)$ such that 
\begin{align}\label{eqn-main-thm-infty}
			\int_\Omega \exp \left(\gamma |I_\alpha f(x)| \right) \dn \leq C.
\end{align}
for every $f\in L^{\frac{n}{\alpha},\infty}(\Omega)$ with $||\, |f|\, ||_{L^{\frac{n}{\alpha},\infty}(\Omega)}\leq 1$.  Moreover, the result is sharp in the sense that the l.h.s.~of \eqref{eqn-main-thm-infty}, with $\gamma\geq\left(\frac{d}{n}\right)\gamma(\alpha)\,\omega_n^{-\frac{n-\alpha}{n}}$, cannot be uniformly bounded for every $f\in L^{\frac{n}{\alpha},\infty}(\Omega)$ with $||\, |f|\, ||_{L^{\frac{n}{\alpha},\infty}(\Omega)}\leq 1$.
	\end{theorem}

For the higher order gradient we have the following analogue of Theorem \ref{main-thm-0}.

\begin{theorem}\label{main-thm-infty-1}
		Let $\Omega$ be an open bounded subset of $\rn$, $n\geq 2$. Let $k \in \mathbb{N}$, $1\leq k<n$.  Let $\nu$ be any positive Borel measure on $\Omega$ which satisfies \eqref{bgc} and \eqref{condition-1} for some $d\in (0,n]$.  Let $\gamma< d \,\omega_n^{\frac{k}{n}}\,\sqrt{\ell^k_n}$.  Then there exists positive constant $C=C(n,k,\Omega, C'_d, \gamma)$ such that 
\begin{align}\label{eqn-ss-lorentz-infty}
\int_\Omega \exp \left( \gamma|u(x)|\right) \dn \leq C,
\end{align}
for every $u\in W_0^k L^{\frac{n}{k},\infty}(\Omega)$ with $||\, |\nabla^k u|\, ||_{L^{\frac{n}{k},\infty}(\Omega)}\leq 1$. Moreover, the result is sharp in the sense that the l.h.s.~of \eqref{eqn-ss-lorentz-infty}, with $\gamma> d \,\omega_n^{\frac{k}{n}}\,\sqrt{\ell^k_n}$, cannot be uniformly bounded for every $u\in W_0^k L^{\frac{n}{k},\infty}(\Omega)$ with $||\, |\nabla^k u|\, ||_{L^{\frac{n}{k},\infty}(\Omega)}\leq 1$.
\end{theorem}

	Finally, for Adams' derivative $D^ku$ we have the following analogue of Theorem \ref{main-thm-1}.	
	
	\begin{theorem}\label{main-thm-infty-2}
			Let $\Omega$ be an open bounded subset of $\rn$, $n\geq 2$. Let $k \in \mathbb{N}$, $1\leq k<n$.  Let $\nu$ be any positive Borel measure on $\Omega$ which satisfies \eqref{bgc} and \eqref{condition-1} for some $d\in (0,n]$.  Let $\gamma<\beta_{n,k,\infty}$ for $\beta_{n,k,\infty}$ as defined in \eqref{beta-n-k}.  Then there exists positive constant $C=C(n,k,\Omega, C'_d, \gamma)$ such that 
\begin{align}\label{eqn-adams-lorentz-infty}
\int_\Omega \exp \left( \gamma|u(x)|\right) \dn \leq C,
\end{align}
for every $u\in W_0^k L^{\frac{n}{k},\infty}(\Omega)$ with $||\, |D^k u|\, ||_{L^{\frac{n}{k},\infty}(\Omega)} \leq 1$. Moreover, the result is sharp in the sense that the l.h.s.~of \eqref{eqn-adams-lorentz-infty}, with $\gamma>\beta_{n,k,\infty}$, cannot be uniformly bounded for every $u\in W_0^k L^{\frac{n}{k},\infty}(\Omega)$ with $||\, |D^k u|\, ||_{L^{\frac{n}{k},\infty}(\Omega)}\leq 1$.
\end{theorem}
	
One observes that in Theorems \ref{main-thm-infty-1} and  \ref{main-thm-infty-2}, in contrast to the case $q<+\infty$ for any of the differential or integral operators, or $q=+\infty$ for the Riesz potential, the validity of the inequality at the endpoint is not settled.  In particular, while our Theorem \ref{main-thm-infty} for the Riesz potential shows a failure at the endpoint which is in agreement with Cianchi's establishment of the failure of the inequality at the endpoint in the first order case regime \cite[Theorem 2.2 on p.~2010]{Cianchi:2008}, the test functions we utilize in the higher order differential case are not sufficient to settle the endpoint.  This motivates
\begin{openquestion}
Can one construct test functions which show the failure of Theorems \ref{main-thm-infty-1} and  \ref{main-thm-infty-2} at their unresolved endpoints, $\gamma=d \,\omega_n^{\frac{k}{n}}\,\sqrt{\ell^k_n}$ and $\gamma=\beta_{n,k,\infty}$, respectively?
\end{openquestion}
This question of the endpoint behavior for these limiting inequalities is related to an interesting underlying phenomena.  In particular, for $q=+\infty$ there is a difference in character of the embedding, as in contrast to the case $q<+\infty$, the usual test function of a cutoff function\footnote{The cutoff function is to ensure compact support in the domain, and handling this error or finding another mechanism is what must be done to establish failure of the endpoint.} times $\log |x|$ is an element of the space $W_0^k L^{\frac{n}{k},\infty}(\Omega)$.  For $q<+\infty$, the typical element of $W_0^k L^{\frac{n}{k},q}(\Omega)$ is a cutoff function times $\log (2+\log|x|)$, which has slightly better behavior.  This alternative direction of improvement to exponential integrability for functions in critical Sobolev spaces has been observed in the literature in the work of K. Hansson \cite[Eqn 3.13 on p.~101]{Hansson} and H. Brezis and S. Wainger \cite[Theorem 2 on p.~781]{BW}, where the relevant functional inequality is the Hansson-Brezis-Wainger embedding:  Let $1<q<+\infty$.  There exists a constant $C>0$ such that
\begin{align}\label{HBW_Lebesgue_Riesz}
\left(\int_{0}^{|\Omega|}\left(\frac{(I_\alpha f)^*(t)}{1+\log\left(\frac{|\Omega|}{t}\right)}\right)^q\;\frac{dt}{t}\right)^{\frac{1}{q}} \leq C\|f\|_{L^{n/\alpha,q}(\Omega)}
\end{align}
for every $f\in L^{\frac{n}{\alpha},q}(\Omega)$.  Here for a measurable function $g$, $g^*$ is the usual non-increasing rearrangement of $g$, see e.g. \cite[Definition 1.4.1 on p.~44]{grafakos}, i.e., for any $t\in (0,\infty)$, we have
	\begin{align}\label{nonincreasing_rearrangement}
		g^*(t)=\inf\{\lambda>0; \, |\{x\in \Omega; \, |g(x)|>\lambda\}|\leq t\}.
	\end{align}
When one takes into account \eqref{representation_inequality}, this translates to an estimate for Sobolev functions as
\begin{align}\label{HBW_Lebesgue_Sobolev}
\left(\int_{0}^{|\Omega|}\left(\frac{u^*(t)}{1+\log\left(\frac{|\Omega|}{t}\right)}\right)^q\;\frac{dt}{t}\right)^{\frac{1}{q}} \leq C'\|\nabla^k u\|_{L^{n/\alpha,q}(\Omega)}
\end{align}
for every $u\in W_0^k L^{\frac{n}{k},q}(\Omega)$.  

We do not know of any work to establish the optimal constant in the inequalities \eqref{HBW_Lebesgue_Riesz} and \eqref{HBW_Lebesgue_Sobolev}.  Yet given our work on extending the optimal constants in the inequalities \eqref{critical} to the trace embeddings \eqref{critical_trace}, it is natural to ask whether there is a trace improvement of the Hansson-Brezis-Wainger embedding.  Indeed, we here prove such an inequality in
	\begin{theorem}[Trace Hansson-Brezis-Wainger]\label{traceHBW}
Let $\Omega$ be an open bounded subset of $\rn$, $n\geq 2$. Let $\alpha \in (0,n)$.  Let $\nu$ be any positive Borel measure on $\Omega$ which satisfies \eqref{bgc} for some $d\in (0,n]$.  Then there exists positive constant $C=C(n,\alpha,\Omega, C'_d, q)$ such that 
\begin{align}\label{bw-eq}
			\left(	\int_{0}^{\nu(\Omega)}\frac{\left((I_\alpha f)_{\nu}^{*}(t)\right)^q}{\left(1+|\log \frac{t}{\nu(\Omega)}|\right)^q}\frac{\dt}{t}\right)^{\frac{1}{q}} \leq C\|f\|_{L^{\frac{n}{\alpha},q}(\Omega)}
		\end{align}	
for every $f\in L^{\frac{n}{\alpha},q}(\Omega)$.
\end{theorem}
\noindent
Note here that $(I_\alpha f)_{\nu}^{*}$ is the rearrangement taken with respect to the Radon measure $\nu$:  For a $\nu$ measurable function $g$ and $t\in (0,\infty)$,
	\begin{align}\label{nonincreasing_rearrangement}
		g_\nu^*(t)=\inf\{\lambda>0; \, \nu(\{x\in \Omega; \, |g(x)|>\lambda\})\leq t\}.
	\end{align}
In the sequel, when there is no possibility of ambiguity, for notational convenience we drop the subscript of $\nu$ in $g_\nu^*(t)$ and write $g^*(t)$ for rearrangements taken with respect to both $\nu$ and the Lebesgue measure.

Finally, let us record the following easy consequence of Theorem \ref{traceHBW} and inequality \eqref{representation_inequality}.
\begin{corollary}\label{diff_traceHBW}
Let $\Omega$ be an open bounded subset of $\rn$, $n\geq 2$. Let $k \in \mathbb{N}$, $1\leq k<n$.  Let $\nu$ be any positive Borel measure on $\Omega$ which satisfies \eqref{bgc} for some $d\in (0,n]$.  Then there exists positive constant $C=C(n,k,\Omega, C'_d, q)$ such that
\begin{align}\label{bw-eq}
			\left(	\int_{0}^{\nu(\Omega)}\frac{\left(u_{\nu}^{*}(t)\right)^q}{\left(1+|\log \frac{t}{\nu(\Omega)}|\right)^q}\frac{\dt}{t}\right)^{\frac{1}{q}} \leq C\|\nabla^k u\|_{L^{\frac{n}{k},q}(\Omega)}
		\end{align}
for every $u\in W_0^k L^{\frac{n}{k},q}(\Omega)$.
	\end{corollary}
\begin{remark}
One also has the inequality for $\|\nabla^k u\|_{L^{\frac{n}{k},q}(\Omega)}$ replaced with $\|D^k u\|_{L^{\frac{n}{k},q}(\Omega)}$.
\end{remark}

The plan of the papers is as follows.  In Section \ref{preliminaries}, we recall some necessary preliminaries, including several Hardy inequalities, the precise definitions of the constants appearing in the paper, and some background on non-increasing rearrangements.  We also recall and prove several lemmas involving non-increasing rearrangements that will be useful in the sequel.  In Section \ref{main_results}, we prove Theorems \ref{main-thm-0},  \ref{main-thm}, and  \ref{main-thm-1}.  In Section \ref{main_results_infty}, we prove Theorems \ref{main-thm-infty}, \ref{main-thm-infty-1}, and  \ref{main-thm-infty-2}.  Finally, in Section \ref{HBW} we prove Theorem \ref{traceHBW} and Corollary \ref{diff_traceHBW}.

	\section{Preliminaries}\label{preliminaries}
	We begin this section by fixing some notation.  Throughout the paper, we assume $\Omega$ is a bounded open set in $\rn$ with dimension $n\geq 2$. Our results are quantified over $d\in (0,n]$ and all non-negative Borel measures satisfying the ball growth conditions \eqref{bgc} and \eqref{condition-1}.  
	
	For $k\in \mathbb{N}$, we denote by
	\begin{align*}
		\nabla^k u(x)=\left\{\frac{\partial^k u}{\partial_{x_{i_1}}\cdots\partial_{x_{i_k}}}\right\}_{i_k, \cdots, i_k\in \mathcal{I}_n} \quad \text{ where } \quad  \mathcal{I}_n=\{1,\cdots,n\},
	\end{align*}
	the tensor consisting of the $n^k$ partial derivatives of $u$ of order $k$ at the point $x$ and $|\nabla^k u|(x)$ denotes the Euclidean norm of this vector in $\mathbb{R}^{n^k}$.
	
	We now recall two Hardy inequalities that will be useful in the sequel.  The first is the classical Hardy inequality, see e.g. \cite[Eqn. (1.3.1) on p.~40]{Mazyabook}:
\begin{lemma}\label{another-hardy}
		Let $1<p<\infty$ and $w>1$. If $\psi \geq 0$ is measurable, then
\begin{align*}
\int_{0}^{\infty} t^{-w} \left(\int_{0}^{t}\psi(s)\ds\right)^p
\dt \leq  \left(\frac{p}{w-1}\right)^p \int_{0}^{\infty}t^{-w}(t\psi(t))^p \dt.
\end{align*} 
\end{lemma}
	
The second is \cite[Lemma 1 on p.~782]{BW} (with $R$ replacing $1$), which we state as	
	\begin{lemma}\label{log-hardy}
		Let $1<p<\infty$ and $R>0$.  If $\psi \geq 0$ is measurable, then
\begin{align*}
\int_{0}^{R}\frac{\left(\int_{t}^{R}\psi(s)\ds\right)^p}{\left(1+|\log \left(\frac{t}{R}\right)|\right)^p}\frac{\dt}{t}\leq \left(\frac{p}{p-1}\right)^p\int_{0}^{R}t^{p-1}\psi^p(t)\dt.
\end{align*} 
\end{lemma}

	We next recall the definition of the Riesz potentials.  For $\alpha \in (0,n]$ and $f \in L^1_{loc}(\mathbb{R}^n)$ we define
\begin{align*}
I_\alpha f(x) := \frac{1}{\gamma(\alpha)} \int_{\mathbb{R}^n} \frac{f(y)}{|x-y|^{n-\alpha}}\;dy,
\end{align*}
where
	\begin{align}
		\gamma(\alpha):=\frac{\pi^{n/2} 2^\alpha \Gamma\left(\alpha/2\right)}{\Gamma\left(\frac{n-\alpha}{2}\right)},
	\end{align}
	see for example, \cite[p.~117]{S}.  It is also useful to use this constant to define each $\alpha \in [0, n)$,
	\begin{equation}
		\widetilde{\gamma}(\alpha)=
		\left\{ \begin{array}{ll}
			\alpha \gamma(\alpha) & \text{ if } \alpha>0\,; \\
			n\omega_{n}  & \text{ if } \alpha =0\,.
		\end{array} \right.
	\end{equation}
	Note that $\widetilde{\gamma}(\alpha)$ is continuous at $\alpha=0$. 
	
	With this preparation, we introduce several representation formulas.  First, we have the formula recorded in \cite[Lemma 2]{Adams1988} (see also the discussion in \cite{SS}). For $k \in \mathbb{N}$ and $u \in C^\infty_c(\mathbb{R}^n)$ we have
	\begin{equation}\label{derv-rep}
		u(x)=
		\left\{ \begin{array}{ll}
			\frac{1}{\gamma(k)} \int_{\mathbb{R}^{n}}|x-y|^{k-n} D^k u(y) \dy & \text{ if } k \text { is even}; \\
			\frac{1}{\widetilde{\gamma}(k-1)} \int_{\mathbb{R}^{n}} |x-y|^{k-1-n}(x-y)\cdot D^k u(y)  \dy & \text{ if } k \text { is odd}.
		\end{array} \right.
	\end{equation}
	
	From \eqref{derv-rep} and using the definition of the Riesz potential, one has the pointwise inequality 
	\begin{equation}\label{pt-wise}
		|u(x)| \leq
		\left\{ \begin{array}{ll}
			I_{k}\left(\left|D^{k} u\right|\right)(x) & \text{ if } k \text { is even}; \\
			\frac{\gamma(k)}{\widetilde{\gamma}(k-1)} I_{k}\left(\left|D^{k} u\right|\right)(x) &\text{ if }  k \text { is odd}.
		\end{array} \right.
	\end{equation}

For the differential operators $D^k$, it will be useful to define constants that depend on the parity of $k$:  For $k \in \mathbb{N}$, $1 \leq k <n$, we define
	\begin{equation}\label{beta-n-k}
		\beta_{n,k,q}=
		\left\{ \begin{array}{ll}
			\left(\frac{d}{n}\right)^{\frac{1}{q^\prime}}\gamma(k)\,\omega_n^{-\frac{n-k}{n}} & \text{ if } k \text { is even}; \\
			\left(\frac{d}{n}\right)^{\frac{1}{q^\prime}}\widetilde{\gamma}(k-1) \, \omega_n^{-\frac{n-k}{n}} &\text{ if }  k \text { is odd}.
		\end{array} \right.
	\end{equation}

	We next recall the definition of a few relevant constants and their relation with the function $\log|x|$.  First, we recall the combinatorial constant which appears in \eqref{combinatorial_constant} (see e.g. \cite{MoriiSatoSawano}):
	\begin{align*}
		\ell^k_n=(k)!\sum_{l=0}^{\lfloor k/2 \rfloor}(k-2l)!(l)!\left(\frac{n-3}{2}+l\right)_l\,\left( \sum_{t=\lceil k/2 \rceil}^{k-l}2^{2t-k+l}\frac{(-1)^t}{2t}\binom{t}{k-t}\binom{k-t}{l}\right)^2
	\end{align*}
	where
	\begin{align*}
		(\mu)_i=
		\left\{ \begin{array}{ll}
			\prod_{j=0}^{i-1}(\mu-j)  & \text{ for } \mu\in\R,\, i\in\N; \\
			1 & \text{ for } \mu\in\R,\, i=0.
		\end{array} \right.
	\end{align*}

This constant makes an appearance in the following lemma which will be useful in the sequel (and whose demonstration follows easily from the computations in \cite{SS}).
	\begin{lemma}\label{log_constants}
		Let $k \in \mathbb{N}$ and $\alpha \in (0,n)$.  For $x\neq 0$ one has
		\begin{align}
			|\nabla^k \log |x| | &= \frac{\sqrt{\ell^k_n}}{|x|^k}, \label{natural}\\
			|\nabla (-\Delta)^{(k-1)/2} \log |x| | &=      \frac{\tilde{\gamma}(k-1)}{n\omega_n} \frac{1}{|x|^k} , \label{adams}\\
			|(-\Delta)^{\alpha/2} \log |x| | &= \frac{\gamma(\alpha)}{n\omega_n} \frac{1}{|x|^{\alpha}} .\label{potential}
		\end{align}
	\end{lemma}
	
The preceding lemma is useful in establishing the following pointwise estimate for $u$ in terms of $\nabla^k u$ given in \cite[Corollary 3.1]{SS}. Let $k\in \mathbb{N}$, $1 \leq k<n$. Then for $u\in C_0^\infty(\rn)$, one has 
	\begin{align}\label{pt-wise-org}
		|u(x)|\leq \frac{\gamma(k)}{n\omega_{n}\sqrt{\ell^k_n}}\, I_k(|\nabla^k u|)(x).
	\end{align}

We next recall the definition of the Lorentz spaces and Sobolev--Lorentz spaces.  For $1\leq p,q\leq \infty$ we denote by $L^{p,q}(\Omega)$ the Lorentz space of those measurable functions $u$ for which the following quantity
	\[
	\|u\|_{L^{p,q}(\Omega)}:= \begin{cases} \left(\int_{0}^{|\Omega|}\left(s^\frac{1}{p}u^*(s)\right)^q \frac{\ds}{s}\right)^{\frac{1}{q}} & \text { if } q<\infty; \\ \underset{0<s<|\Omega|}{\sup} s^{\frac{1}{p}} u^*(s) & \text { if } q=\infty,\end{cases}
	\]
	is finite.  

We also recall an alternative formulation of the Lorentz spaces given in \cite[Proposition~1.4.9 on p~53]{grafakos}:
		\[
		\|u\|_{L^{p,q}(X)}:= \begin{cases} p^{\frac{1}{q}}\left(\int_{0}^{\infty}\big[d^X_u(s)^{\frac{1}{p}}s\big]^q \frac{\ds}{s}\right)^{\frac{1}{q}} & \text { if } q<\infty; \\ \underset{s>0}{\sup} \,\, s\,d^X_u(s)^{\frac{1}{p}} & \text { if } q=\infty,\end{cases}
		\]
		where the distribution function $d_u^X$ is defined on $[0,\infty)$ as follows:
\begin{align*}
			d_u^X(s):=\left|\{x\in X\, :\, |u(x)|>s\}\right|.
\end{align*}
	
	For any integer $k\geq 1$ and any  $1\leq p,q\leq \infty$, we define
	\begin{align*}
		W_0^k L^{p,q}(\Omega)&=\{u:\, u \text{ is a real-valued function in } \Omega \text{ whose continuation by } 0 \text{ out }\\
		& \text{ side }\Omega \text{ is }k\text{-times weakly differentiable in the whole space } \rn, \text{ and }\\& \|\, |\nabla^k u|\, \|_{L^{p,q}(\Omega)}<\infty\},
	\end{align*}
	where $|\nabla^k u|$ is the Euclidean norm of $\nabla^k u$ defined earlier. 
	
Let us observe that for any $k\geq 1$, and $p\geq 1$, we have
	\begin{align*}
		L^{p,p}(\Omega)=L^p(\Omega) \quad \text{ and hence }\quad W_0^k L^{p,p}(\Omega)=W^{k,p}_0(\Omega).
	\end{align*}

In connection to the sharpness of the constants, a useful result to record is
	\begin{lemma}\label{weak-type_norm}
		Let $\alpha \in (0,n)$, $c>0$.  Then for $x_0 \in \Omega$,
		\begin{align*}
			\left\| \frac{c}{|\cdot-x_0|^\alpha}\right\|_{L^{n/\alpha,\infty}(\Omega)} \leq c\omega_n^{\alpha/n}.
		\end{align*}
	\end{lemma}
	\begin{proof}
		The claim follows from the computation
		\begin{align*}
			t\left| \left\{ x \in \Omega \colon \frac{c}{|x-x_0|^\alpha} >t\right \}\right|^{\alpha/n} &\leq t\left|B\left(0,\left(\frac{c}{t}\right)^{1/\alpha}\right)\right|^{\alpha/n} \\
			&= c \omega_n^{\alpha/n}.
		\end{align*}
	\end{proof}
	
Another useful observation that will be helpful in dealing with the quasi-norm in pieces of the domain is
	\begin{lemma}\label{lor-break}
		Let $1\leq p< \infty$, and $1\leq q\leq \infty$. Assume $A$, and $B$ be measurable subsets of $\rn$ such that $A\cap B=\emptyset$. Then for any measurable function $u$ on $A\cup B$, there holds
		\begin{align}\label{eq-lor-break}
			\|u\|_{L^{p,q}(A\cup B)}\leq \|u\|_{L^{p,q}(A)} +\|u\|_{L^{p,q}(B)}.
		\end{align}
	\end{lemma}
	\begin{proof}
		Let $X$ be a measurable subset of $\rn$.  For any $s\in [0,\infty)$, we notice
		\begin{align*}
			\{x\in A\cup B\, :\, |u(x)|>s\} =\{x\in A\, :\, |u(x)|>s\}\cup\{x\in  B\, :\, |u(x)|>s\}.
		\end{align*}
		Moreover, the above union is disjoint due to the fact $A\cap B=\emptyset$. Therefore, we have
		\begin{align}\label{lor-break-st-1}
			d_u^{A\cup B}(s)= d_u^{A}(s) + d_u^{B}(s).
		\end{align}
		
		Given that $p\geq 1$, using \eqref{lor-break-st-1}, we have 
		\begin{align}\label{lor-break-st-2}
			d_u^{A\cup B}(s)^{\frac{1}{p}} \leq d_u^{A}(s)^{\frac{1}{p}}+d_u^{B}(s)^{\frac{1}{p}} 
		\end{align}
		for any $0\leq s<\infty$. Now suppose the exponent $q=\infty$. Then multiplying both side of \eqref{lor-break-st-2} by $s$ and taking supremum over it, we have
		\begin{align*}
			\|u\|_{L^{p,\infty}(A\cup B)} \leq 	\|u\|_{L^{p,\infty}(A)} + 	\|u\|_{L^{p,\infty}(B)}.
		\end{align*}
		
		Now consider the case $1\leq q<\infty$. Then, observe that
		\begin{align*}
			\|u\|_{L^{p,q}(A\cup B)} = p^{\frac{1}{q}}\left(\int_{0}^{\infty}\big[d^{A\cup B}_u(s)^{\frac{1}{p}}s^{1-\frac{1}{q}}\big]^q \ds\right)^{\frac{1}{q}}=p^{\frac{1}{q}}\|d^{A\cup B}_u(s)^{\frac{1}{p}}s^{1-\frac{1}{q}}\|_{L^q(0,\infty)}.
		\end{align*} 
		Hence, using \eqref{lor-break-st-2} we note
		\begin{align*}
			\|d^{A\cup B}_u(s)^{\frac{1}{p}}s^{1-\frac{1}{q}}\|_{L^q(0,\infty)}&\leq \|\left(d_u^{A}(s)^{\frac{1}{p}}+d_u^{B}(s)^{\frac{1}{p}}\right)s^{1-\frac{1}{q}}\|_{L^q(0,\infty)}\\ & \leq \|d^{A}_u(s)^{\frac{1}{p}}s^{1-\frac{1}{q}}\|_{L^q(0,\infty)}+\|d^{B}_u(s)^{\frac{1}{p}}s^{1-\frac{1}{q}}\|_{L^q(0,\infty)},
		\end{align*}
		where in the last line we use the triangle inequality for the space $L^q(0,\infty)$ with  $1\leq q<\infty$. Finally, multiplication by $p^{\frac{1}{q}}$ and using the equivalent definition of Lorentz space quasi-norm, we deduce \eqref{eq-lor-break}.
	\end{proof}

	Suppose that $(\rn, \mathcal{L}^n)$ and $(\Omega, \nu)$ are measure spaces and $T$ is an integral operator defined as
	\begin{align*}
		Tf(x)=\int_{\rn} K(x,y) f(y) \dy \quad \text{ for } x\in \Omega,
	\end{align*}
	where $K:\Omega\times \rn \rightarrow [-\infty,+\infty]$ be  $\nu\times\mathcal{L}^n$-measurable and $f:\rn\rightarrow [0,+\infty]$ be $\mathcal{L}^n$-measurable functions. Now, let us define the partial non-increasing rearrangements below: for $t>0$, we define
	\begin{align*}
		&k_x^*(t)=\inf_{s\geq 0}\{s\,:\, | \{y\in \rn: |K(x,y)|>s\}|\leq t\}	\quad \text{ for } x\in \Omega,\\
		&k_y^*(t)=\inf_{s\geq 0}\{s\,:\, \nu( \{x\in \Omega: |K(x,y)|>s\})\leq t\}	\quad \text{ for } y\in \rn.
	\end{align*}
	Then maximal non-increasing rearrangements are defined as
	\begin{align*}
		k_1^*(t)=\sup_{x\in \Omega} k_x^*(t)	\quad \text{ and } \quad k_2^*(t)=\sup_{y\in \rn} k_y^*(t),
	\end{align*}
	for $t>0$. Similarly, for $t>0$ we define
	\begin{align*}
		(Tf)^*(t)=\inf_{s\geq 0}\{s\,:\, \nu( \{x\in \Omega: |Tf(x)|>s\})\leq t\},
	\end{align*}
	and
	\begin{align*}
		(Tf)^{**}(t)=\frac{1}{t}\int_0^t (Tf)^*(s) \ds. 
	\end{align*}

When $K(x,y)=\frac{1}{\gamma(\alpha)}|x-y|^{\alpha-n}$ for $0<\alpha <n$, $Tf(x)$ is just the Riesz potential defined above.  For this operator one computes, for $t>0$,
	\begin{align*}
		k_x^*(t)=\frac{1}{\gamma(\alpha)}\omega_n^{\frac{n-\alpha}{n}} t^{-\frac{n-\alpha}{n}}, \quad \text{ and hence } \quad  k_1^*(t)=\frac{1}{\gamma(\alpha)}\omega_n^{\frac{n-\alpha}{n}} t^{-\frac{n-\alpha}{n}},
	\end{align*}
	where $\omega_{n}$ is the volume of unit ball in $\rn$.

	Now, under the assumption that a non-negative Borel measure $\nu$ satisfies \eqref{bgc}, for any fixed $y\in \rn$, we have
\begin{align*}
\nu( \{x\in \Omega: |K(x,y)|>s\}) &= \nu(\Omega\cap B_{{(s\gamma(\alpha))}^{-\frac{1}{n-\alpha}}}(y)) \\
&\leq C'_d(s\gamma(\alpha))^{-\frac{d}{n-\alpha}}.
\end{align*}
In particular, the value $s=t^{-\frac{n-\alpha}{d}} {C_d'}^{\frac{n-\alpha}{d}}/\gamma(\alpha)$ is an element of the set
\begin{align*}
\{s\,:\, \nu( \{x\in \Omega: |K(x,y)|>s\})\leq t\}.
\end{align*}	
Therefore, 
\begin{align*}
		k_y^*(t)&=\inf_{s\geq 0}\{s\,:\, \nu( \{x\in \Omega: |K(x,y)|>s\})\leq t\}\\
		&\leq \frac{1}{\gamma(\alpha)}{C_d'}^{\frac{n-\alpha}{d}} t^{-\frac{n-\alpha}{d}}.
\end{align*}
	and so we have
	\begin{align*}
		k_2^*(t)=\sup_{y\in \rn} k_y^*(t)\leq \frac{1}{\gamma(\alpha)}{C_d'}^{\frac{n-\alpha}{d}} t^{-\frac{n-\alpha}{d}}.
	\end{align*}

	A key result that relates the rearrangement of a convolution to the rearrangements of the convolved functions is due to O'Neil \cite{oneil}. Here we use a modified version which can be found in \cite[Lemma 5]{FontanaMorpurgo:na}.
	\begin{lemma}\label{oneil-lemma}
		Assume $0<\alpha<n$ and $0<d\leq n$. Let $f:\Omega\rightarrow [0,+\infty]$ be $\mathcal{L}^n$-measurable function. Then for any
\begin{align}\label{eqn-oneil-lemma-exponent}
			\max\left\{ 1, \frac{n-d}{\alpha} \right\} <r<\frac{n}{\alpha}, \quad \text{ and defining } \quad \theta=\frac{rd}{n-\alpha r},
		\end{align}
there exists a constant $C=C(\alpha,n,d,C_d',r)>0$ such that
		\begin{align}\label{eqn-oneil-lemma-rearrangement}
			(I_\alpha f)^{**}(t)\leq C\max\{\tau^{-\frac{d}{n\theta}},t^{-\frac{1}{\theta}}\}\int_0^\tau f^*(u)u^{-1+\frac{1}{r}} du+\frac{1}{\gamma(\alpha)}\omega_n^{\frac{n-\alpha}{n}}\int_\tau^{|\Omega|}  f^*(u)  u^{-\frac{n-\alpha}{n}} du
		\end{align}
 for all $\tau\, ,\, t>0$.
	\end{lemma}

	We also make use of an extension of Moser's original one-dimensional lemma established in \cite{Adams1988}, the following
	\begin{lemma}\label{adams-lemma}
		Let $1<q<\infty$ and define where $q^\prime=q/(q-1)$.  Let $a(s,t)$ be a non-negative measurable function on $[0,\infty)\times [0,\infty)$ such that for almost every $s,t>0$
		\begin{equation}\label{lem-1-eq-1}
			a(s,t) \leq 1, \quad \text{ when } 0\leq s<t,
		\end{equation}
		\begin{equation}\label{lem-1-eq-2}
			\sup_{t>0}\bigg(\int_{t}^{\infty}a(s,t)^{q^\prime}\ds\bigg)^{1/{q^\prime}}=b<\infty.
		\end{equation}
		Then there is a constant $c=c(q,b)$ such that for every non-negative $\phi:(0,\infty) \to [0,\infty]$ which satisfies 
		\begin{equation}\label{lem-1-eq-3}
			\int_{0}^{\infty}\phi(s)^q \ds\leq 1,
		\end{equation}
		one has
		\begin{equation}\label{lem-1-eq-4}
			\int_0^\infty e^{-F(t)}\dt\leq c,
		\end{equation}
		where
		\begin{equation}\label{lem-1-eq-5}
			F(t)=t-\bigg(\int_{0}^\infty a(s,t)\phi(s)\ds\bigg)^{q^\prime}.
		\end{equation}
		
	\end{lemma}

	\section{Proofs of the Main Results}\label{main_results}
In this section, we prove Theorems \ref{main-thm-0},  \ref{main-thm}, and  \ref{main-thm-1}.  Following the program of Adams from \cite{Adams1988}, we begin with the establishment of sharp potential estimates in the
	\begin{proof}[Proof of Theorem \ref{main-thm}]
		It suffices to establish the estimate for $\beta = \left(\frac{d}{n}\right)^{\frac{1}{q^\prime}}\gamma(\alpha)\,\omega_n^{-\frac{n-\alpha}{n}}$.  Let $r$ be any admissible exponent in the range specified in \eqref{eqn-oneil-lemma-exponent} and choose $\tau=t^{\frac{n}{d}}$ in \eqref{eqn-oneil-lemma-rearrangement}. Then we have
		\begin{align*}
			( I_\alpha f)^*(t)&\leq (I_\alpha f)^{**}(t)\leq Ct^{-\frac{1}{\theta}}\int_0^{t^{\frac{n}{d}}} f^*(u)u^{-1+\frac{1}{r}} du+\frac{1}{\gamma(\alpha)}\omega_n^{\frac{n-\alpha}{n}}\int_{t^{\frac{n}{d}}}^{|\Omega|}  f^*(u)  u^{-\frac{n-\alpha}{n}} du\\&=Ct^{-\frac{1}{\theta}}\int_0^{t} f^*(s^{\frac{n}{d}})s^{-1+\frac{n}{dr}} ds+\frac{n}{d\gamma(\alpha)}\omega_n^{\frac{n-\alpha}{n}}\int_{t}^{|\Omega|^{\frac{d}{n}}}  f^*(s^{\frac{n}{d}})  s^{-1+\frac{\alpha}{d}} ds
		\end{align*}
		Now, take $t_1=\max\{\nu(\Omega), |\Omega|^{\frac{d}{n}}\}$ and we deduce
		\begin{align*}
			&\int_\Omega \exp \left( \left(\frac{d}{n}\right)^{\frac{1}{q^\prime}}\gamma(\alpha)\omega_n^{-\frac{n-\alpha}{n}}|I_\alpha f(x)| \right)^{q^\prime} \dn \\&=	\int_0^{\nu(\Omega)} \exp \left( \left(\frac{d}{n}\right)^{\frac{1}{q^\prime}}\gamma(\alpha)\omega_n^{-\frac{n-\alpha}{n}}(I_\alpha f)^*(t) \right)^{q^\prime} \dt\\& \leq \int_0^{\nu(\Omega)} \exp \left( \left(\frac{d}{n}\right)^{\frac{1}{q^\prime}}\gamma(\alpha)\omega_n^{-\frac{n-\alpha}{n}}(I_\alpha f)^{**}(t) \right)^{q^\prime} \dt\\&\leq \int_0^{t_1} \exp \left( Ht^{-\frac{1}{\theta}}\int_0^{t} f^*(s^{\frac{n}{d}})s^{-1+\frac{n}{dr}} \ds + \left(\frac{n}{d}\right)^{\frac{1}{q}}\int_{t}^{t_1}  f^*(s^{\frac{n}{d}})  s^{-1+\frac{\alpha}{d}} \ds \right)^{q^\prime} \dt,
		\end{align*}
		where $H$ is some suitable generic constant. Next, we apply the changes of variables $s=e^{-x},t=e^{-y}$ and define $y_1=-\log t_1$. This yields
		\begin{align*}
			&\int_0^{t_1} \exp  \left( Ht^{-\frac{1}{\theta}}\int_0^{t} f^*(s^{\frac{n}{d}})s^{-1+\frac{n}{dr}} \ds + \left(\frac{n}{d}\right)^{\frac{1}{q}}\int_{t}^{t_1}  f^*(s^{\frac{n}{d}})  s^{-1+\frac{\alpha}{d}} \ds \right)^{q^\prime} dt\\&\leq \int_{y_1}^{\infty} \exp\left[ \left( He^{\frac{y}{\theta}}\int_y^{\infty} f^*(e^{-\frac{nx}{d}})e^{-\frac{nx}{dr}} \dx +\left(\frac{n}{d}\right)^{\frac{1}{q}} \int_{y_1}^{y}  f^*(e^{-\frac{nx}{d}})  e^{-\frac{\alpha x}{d}} \dx \right)^{q^\prime} -y\right]\dy.
		\end{align*}
		We next define $\phi(x)=\left(\frac{n}{d}\right)^{\frac{1}{q}}  f^*(e^{-\frac{nx}{d}})  e^{-\frac{\alpha x}{d}}$ on $[y_1,\infty)$.  By the change of variables $s=e^{-\frac{nx}{d}}$, we have
		\begin{align*}
			\|\phi\|_{L^q(y_1,\infty)}^q=\frac{n}{d} \int_{y_1}^{\infty}  f^*(e^{-\frac{nx}{d}})  e^{-\frac{\alpha x}{d}} \dx \leq \int_0^{|\Omega|}{f^*}^q(s)s^{\frac{\alpha q}{n}-1}\ds=\|\,|f|\,\|_{L^{\frac{n}{\alpha}, q}(\Omega)}\leq 1.
		\end{align*}
		
		Therefore, to establish \eqref{eqn-main-thm}, we need to verify
		\begin{align*}
			\int_{y_1}^{\infty} \exp\left[ \left( H\int_y^{\infty} \phi(x) e^{-\frac{nx}{dr}+\frac{\alpha x}{d}+\frac{y}{\theta}} \dx + \int_{y_1}^{y}  \phi(x) \dx \right)^{q^\prime} -y\right]dy \leq C,
		\end{align*}
		where $H$ is a suitable constant. 
		
		Define
		\[
		g(x,y)= \begin{cases}1 & \text { if } y_1\leq x\leq y; \\ H e^{-\frac{nx}{dr}+\frac{\alpha x}{d}+\frac{y}{\theta}}& \text { if } y_1\leq y<x<\infty; \\ 0 & \text{ elsewhere},\end{cases}
		\]
		and 
		\[
		F(y)= \begin{cases}y- \left(\int_{y_1}^\infty g(x,y)\phi(x) \dx\right)^{q^\prime} & \text { if } y_1\leq y <\infty; \\ 0 & \text { elsewhere}.\end{cases}
		\]
		We claim $a=g$ and $\phi$ satisfy the conditions of Lemma \ref{adams-lemma}.   The inequality \eqref{lem-1-eq-1} follows easily from the definition of $g$, while the preceding computation shows $\phi$ satisfies \eqref{lem-1-eq-3}. Concerning \eqref{lem-1-eq-2}, this follows from the computation
		\begin{align*}
			&\sup_{y>0}\bigg(\int_{y}^{\infty}H^{q^\prime} e^{-\frac{nq^\prime x}{dr}+\frac{\alpha q^\prime x}{d}+\frac{q^\prime y}{\theta}}\dx\bigg)^{1/{q^\prime}}\\&=H^{q^\prime}\sup_{y>0}\bigg(\int_{y}^{\infty} e^{\frac{q^\prime (y-x)}{\theta}}\dx\bigg)^{1/{q^\prime}}\\&=\frac{\theta}{q^\prime}\,H^{q^\prime}<\infty,
		\end{align*}
		which makes use of the identity $\theta=\frac{rd}{n-\alpha r}$  from \eqref{eqn-oneil-lemma-exponent}. Therefore \eqref{eqn-main-thm} follows from an application of Lemma \ref{adams-lemma}.
		
		We next show the optimality of the constant.  In fact, it would suffice to show that one cannot improve the constant for Theorem \ref{main-thm-0} or Theorem \ref{main-thm-1}, since an improvement of the constant in this inequality would yield an improvement of the constant in both of those inequalities, which would result in a contradiction.  However, the test functions to show optimality in this inequality turn out to be useful in the case $q=+\infty$, allowing us to resolve the analogous sharpness of the constant for the Riesz potential, but not for the differential operators, and so we here give the construction.
		
		  Without loss of generality, we let $\Omega=B$, the unit ball centered at the origin. Otherwise, up to rescaling and translating we can assume $B\subset\subset \Omega$. Let $B_r$ be the ball of radius $r$ with $0<r<1$ and centered at the origin. Assume that for some $\beta>0$, there holds
		\begin{align}\label{eqn-th2.5-1}
			\int_B \exp \left(\frac{\beta|I_\alpha f(x)|}{\|\, |f|\, \|_{L^{\frac{n}{\alpha},q}(B)}} \right)^{q^\prime} \dn \leq C
		\end{align}
		for every $f\in L^{\frac{n}{\alpha},q}(B)$.	Now, for $0<r<1$ and $x\in B$, let us define
		\begin{align*}
			f_r(x)=
			\left\{ \begin{array}{ll}
				\frac{|x|^{-\alpha}}{n\omega_{n}\log(1/r)}  & \text{ for } r<|x|<1; \\
				0 & \text{ otherwise}.
			\end{array} \right.
		\end{align*}
We then follow the computation in \cite[Theorem 2]{Adams1988}, with some additional details here for the convenience of the reader.  In the region $|x|\leq r$, we have
		\begin{align*}
			\gamma(\alpha)&I_{\alpha}f_r(x)=(n\omega_{n})^{-1}(\ln 1/r)^{-1}\int_{r<|y|<1}|x-y|^{\alpha-n}|y|^{-\alpha} \dy\\&=(n\omega_{n})^{-1}(\ln 1/r)^{-1}\bigg[\int_{|y|\leq 1}|x-y|^{\alpha-n}|y|^{-\alpha}\dy-\int_{|y|\leq r}|x-y|^{\alpha-n}|y|^{-\alpha}\dy\bigg]\\&=(\ln 1/r)^{-1}\left[\I(1)-\I(r)\right],
		\end{align*}
		where for any radius $R>0$, the integral is defined as
		\begin{align*}
			\I(R):=\frac{1}{n\omega_{n}} \int_{|z|\leq R}|x-z|^{\alpha-n}|z|^{-\alpha} \dz.
		\end{align*}
		
		We will evaluate the integral $\I(R)$, a computation which goes back to Fuglede \cite[p.~7]{Fuglede}. Introducing spherical coordinates, we write $\dz=|z|^{n-1} {\rm d}|z| \, {\rm d} w$, where $w$ is in the sphere. Now, because of homogeneity, and writing $t=|x| /|z|$, we have		
		\begin{align*}
			\I(R)=	\frac{1}{n\omega_{n}} \int_{|z|\leq R}|x-z|^{\alpha-n}|z|^{-\alpha} \dz=\int_{\frac{\rho}{R}}^{\infty} t^{-1} u_{\alpha}(t) \dt,	
		\end{align*}
		where $\rho=|x|$, and where $u_{x}(l)$ denotes the \textit{Riesz potential} of order $\alpha$ of the uniform distribution of unit mass on the unit sphere in $\rn$, evaluated at a point of distance $t$ from the origin. One observes that $u_{\alpha}(t)$ is differentiable for $t>1$ and for $0 \leq t<1$, and integrable over a neighbourhood of $t=1$. Moreover,
		\begin{align*}
			u_{\alpha}(0)=1 ; \quad u_{\alpha}^{\prime}(0)=0 ; \quad u_{\alpha}(t)=O\left(t^{\alpha-n}\right), \quad \text{ for }t \rightarrow+\infty.
		\end{align*}

		Therefore, the function $v_{\alpha}$ defined by		
		$$
		v_{\alpha}(t)= \begin{cases}t^{-1} u_{\alpha}(t) & \text { for } t>1; \\ t^{-1}\left(u_{\alpha}(t)-1\right) & \text { for } 0<t<1,\end{cases}
		$$
		is bounded near $t=0$ and integrable over $(0,+\infty)$. We now obtain
		\begin{align*}
			\I(R)=\log \frac{R}{\rho}+V\left(\frac{\rho}{R}\right),
		\end{align*}
		where $V(t)=\int_{t}^{\infty} v_{\alpha}(s) \ds$ is a continuous function and admits the limit
		\begin{align*}
			\lim_{t\rightarrow 0} V(t) = \int_{0}^{\infty} v_{\alpha}(t) \dt.
		\end{align*}
		
Therefore,
		\begin{align*}
			\gamma(\alpha)I_{\alpha}f_r(x)&= (\ln 1/r)^{-1} \bigg[\ln \frac{1}{|x|}+V\left(|x|\right)-\ln \frac{r}{|x|}-V\left(\frac{|x|}{r}\right)\bigg]\\&=(\ln 1/r)^{-1} \bigg[(\ln 1/r)+\int_{|x|}^{|x|/r} v_{\alpha}(t) \dt\bigg].
		\end{align*}
		Now, for every $\epsilon>0$ there exists sufficiently small $r_0>0$ such that for all $0<r<r_0$, there holds
		\begin{align*}
			\bigg|(\ln 1/r)^{-1}	\int_{|x|}^{|x|/r} v_{\alpha}(t) \dt\bigg|\leq (\ln 1/r_0)^{-1}\int_{0}^{\infty} |v_{\alpha}(t)| \dt <\epsilon .
		\end{align*}
In particular,
		\begin{align*}
			\gamma(\alpha)I_{\alpha}f_r(x)\geq 1-\epsilon, \quad \text{ whenever }\quad |x|\leq r \quad \text{and} \quad \forall r\leq r_0.
		\end{align*}
		This gives that for any $\epsilon>0$, there exist $r_0=r_0(\epsilon)$, such that for every $0<r<r_0$, there holds
		\begin{align*}
			\gamma(\alpha)	I_\alpha [(1-\epsilon)^{-1}f_r](x)\geq 1, \quad \text{ whenever }\,\, |x|\leq r<r_0.
		\end{align*}
		
		Now by the definition of non-increasing rearrangement, we have: For $s>0$
		\begin{align*}
			(1-\epsilon)^{-1}f_r^*(s)=
			\left\{ \begin{array}{ll}
				0 & \text{ if } s\geq \omega_n(1-r^n); \\
				\frac{(1-\epsilon)^{-1}}{n\omega_{n}\log(1/r)} \frac{1}{(s\omega_{n}^{-1}+r^n)^{\frac{\alpha}{n}}} &\text{ if } s<\omega_n(1-r^n).
			\end{array} \right.
		\end{align*}
Utilizing this we can estimate
		\begin{align*}
			&\|\, |(1-\epsilon)^{-1}f_r|\, \|^{q}_{L^{\frac{n}{\alpha},q}(B)}\\&=\int_{0}^{\omega_{n}}s^{\frac{\alpha q}{n}-1} {\left( (1-\epsilon)^{-1}f_r\right)^*}^q(s) \ds\\&= \frac{(1-\epsilon)^{-q}}{n^q\omega^q_{n}(\log(1/r))^q} \int_{0}^{\omega_{n}(1-r^n)} \frac{s^{\frac{\alpha q}{n}-1}}{(s\omega_{n}^{-1}+r^n)^{\frac{\alpha q}{n}}} \ds \\&= \frac{\omega^\frac{\alpha q}{n}_{n}(1-\epsilon)^{-q}}{n^q\omega^q_{n}(\log(1/r))^q} \int_{0}^{r^{-n}-1} \frac{t^{\frac{\alpha q}{n}-1}}{(t+1)^{\frac{\alpha q}{n}}} \dt \\& \leq \frac{\omega^\frac{\alpha q}{n}_{n}(1-\epsilon)^{-q}}{n^q\omega^q_{n}(\log(1/r))^q} \left[ \int_0^1 t^{\frac{\alpha q}{n}-1}\dt +\int_1^{r^{-n}} \frac{1}{t} \dt\right]\\&=\frac{\omega^\frac{(\alpha-n) q}{n}_{n}(1-\epsilon)^{-q}}{n^{q-1}(\log(1/r))^q} \left[ \frac{1}{\alpha q} +  \log(1/r)\right] \\&\leq n^{1-q} \omega^\frac{(\alpha-n) q}{n}_{n}(1-\epsilon)^{-q} (1+\delta)^{\frac{q}{q^\prime}} (\log(1/r))^{1-q},
		\end{align*}
		for every $\delta>0$ and $r\leq R(r_0,\delta)$. Hence, using this and making use of \eqref{eqn-1-th2.6} we deduce
		\begin{align*}
			C&\int_B \exp \left(\frac{\beta|I_\alpha[(1-\epsilon)^{-1}f_r](x)|}{\|\, |(1-\epsilon)^{-1}f_r|\, \|_{L^{\frac{n}{\alpha},q}(B)}} \right)^{q^\prime} \dn  \\& \geq 	\int_{B_{r}} \exp \left(  \frac{\beta\, \gamma^{-1}(\alpha)}{n^{-\frac{1}{q^\prime}} \omega^\frac{(\alpha-n) }{n}_{n}(1-\epsilon)^{-1} (1+\delta)^{\frac{1}{q^\prime}} (\log(1/r))^{-\frac{1}{q^\prime}}} \right)^{q^\prime} \dn \\& \geq 	C_0 \,r^d \exp \left(  \frac{\beta^{q^\prime}\, \gamma^{-{q^\prime}}(\alpha)}{n^{-1} \omega^\frac{(\alpha-n){q^\prime} }{n}_{n}(1-\epsilon)^{-{q^\prime}} (1+\delta)}\, \log(1/r)\right)\\&= C_0 \left(  \frac{1}{r}\right)^{\frac{\beta^{q^\prime}\, \gamma^{-{q^\prime}}(\alpha)}{n^{-1} \omega^\frac{(\alpha-n){q^\prime} }{n}_{n}(1-\epsilon)^{-{q^\prime}} (1+\delta)}-d}.
		\end{align*}
		
To ensure finiteness of this quantity as $r\to 0$, we must have
		\begin{align*}
			&\frac{\beta^{q^\prime}\, \gamma^{-{q^\prime}}(\alpha)}{n^{-1} \omega^\frac{(\alpha-n){q^\prime} }{n}_{n}(1-\epsilon)^{-{q^\prime}} (1+\delta)}-d \leq 0 \\& \implies \beta \leq (1+\delta)^{\frac{1}{q^\prime}}(1-\epsilon)^{-1}\left(\frac{d}{n}\right)^{\frac{1}{q^\prime}}\gamma(\alpha)\,\omega_n^{-\frac{n-\alpha}{n}}.
		\end{align*}
		Since, $\delta$ and $\epsilon$ are arbitrary, we first send $\epsilon \to 0$ and then $\delta \to 0$ to obtain 
		\begin{align*}
			\beta \leq \left(\frac{d}{n}\right)^{\frac{1}{q^\prime}}\gamma(\alpha)\,\omega_n^{-\frac{n-\alpha}{n}},
		\end{align*}
		which is the required result.
	\end{proof}

We next give the 	
	
	\begin{proof}[Proof of Theorem \ref{main-thm-0}]
		To establish the positive result, it suffices to prove the inequality for $\beta = d^{\frac{1}{q^\prime}} n^{\frac{1}{q}}\omega_n^{\frac{k}{n}}\sqrt{\ell^k_n}$.  From \eqref{pt-wise-org}, we have
		\begin{align*}
			d^{\frac{1}{q^\prime}} n^{\frac{1}{q}}\,\omega_n^{\frac{k}{n}}\,\sqrt{\ell^k_n}\,|u(x)|	\leq  \left(\frac{d}{n}\right)^{\frac{1}{q^\prime}}\gamma(k)\,\omega_n^{-\frac{n-k}{n}}\,I_k \left( |\nabla^k u| \right) (x),
		\end{align*}
		and the claimed inequality follows from the inequality established in Theorem~\ref{main-thm}.
		
We next turn our attention to optimality.  Without loss of generality, by translation and dilation, we assume $\Omega=B_2$, a ball of radius $2$ and centered at the origin.  Assume that for some $\beta$, there holds 
		\begin{align}\label{eqn-1-th2.6}
			\int_{B_2} \exp \left(  \frac{\beta|u(x)|}{\|\, |\nabla^k u|\, \|_{L^{\frac{n}{k},q}(B_2)}} \right)^{q^\prime} \dn \leq C,
		\end{align}
		for every $u\in W_0^k L^{\frac{n}{k},q}(B_2)$. We want to show that $\beta \leq d^{\frac{1}{q^\prime}} n^{\frac{1}{q}}\,\omega_n^{\frac{k}{n}}\,\sqrt{\ell^k_n}$, and to this end we consider the test functions $\{u_\epsilon\}_{\epsilon>0}\subset C_c^\infty(\rn)$, described in \cite[Proposition 2.1]{SS}. They satisfy
		\begin{align}
			&u_\epsilon(x)=\log (1/|x|) \text{ on } B_1\setminus B_\epsilon, \label{a}\\
			&\|u_\epsilon\|_{L^\infty(\rn)}=u_\epsilon(0)=\log(1/\epsilon)+O(1),\label{b}\\
			&\text{supp}(u_\epsilon)\subset B_2,\label{c}\\
			&\|\,|\nabla^k u_\epsilon|\,\|_{L^\infty(B_\epsilon)}=O(\epsilon^{-k}), \quad 1\leq k \leq n,\label{d}\\
			&\|\,|\nabla^k u_\epsilon|\,\|_{L^\infty(B_2\setminus B_1)}= O(1), \quad 1\leq k \leq n.\label{e}
		\end{align}

		Here we write $C$ as a generic constant everywhere. From the definition of non-increasing rearrangement \eqref{nonincreasing_rearrangement}, we have that for two non-negative functions $f$ and $g$, if $f(x)\leq g(x)$ for $x\in \Omega$, there holds
		\begin{align*}
			f^*(s)\leq g^*(s) \quad \text{ for } s>0.
		\end{align*}

		Now from \eqref{d}, we have $(\chi_{B_\epsilon}|\nabla^k u_\epsilon|)^*(s)\leq (C \chi_{B_\epsilon}\epsilon^{-k})^*(s)$ for $s>0$. Meanwhile
		\begin{equation*}
			(C\chi_{B_\epsilon} \epsilon^{-k})^*(s)=
			\left\{ \begin{array}{ll}
				0 & \text{ if } s\geq \omega_{n}\epsilon^n; \\
				C\epsilon^{-k} &\text{ if } s<\omega_{n}\epsilon^n.
			\end{array} \right.
		\end{equation*}
		Therefore, we find
		\begin{align*}
			\|\,|\nabla^k u_\epsilon|\,\|_{L^{\frac{n}{k},q}(B_\epsilon)}^q &=	\int_0^{\epsilon^n \omega_n} s^{\frac{kq}{n}-1} {|\nabla^k u_\epsilon|^*}^q(s)\ds\\&\leq C \epsilon^{-kq} \int_0^{\epsilon^n \omega_n} s^{\frac{kq}{n}-1} \ds = O(1).
		\end{align*}

		Similarly using \eqref{e}, we observe $(\chi_{B_2\setminus B_1}|\nabla^k u_\epsilon|)^*(s)\leq (C\chi_{B_2\setminus B_1})^*(s)$ for $s>0$, and we have
		\begin{equation*}
			(C\chi_{B_2\setminus B_1} )^*(s)=
			\left\{ \begin{array}{ll}
				0 & \text{ if } s\geq \omega_n(2^n-1); \\
				C &\text{ if } s<\omega_n(2^n-1).
			\end{array} \right.
		\end{equation*}
		Hence, using this we find
		\begin{align*}
			\|\,|\nabla^k u_\epsilon|\,\|_{L^{\frac{n}{k},q}(B_2\setminus B_1)}^q&=	\int_0^{\omega_n(2^n-1)} s^{\frac{kq}{n}-1} {|\nabla^k u_\epsilon|^*}^q(s)\ds\\&\leq C \int_0^{\omega_n(2^n-1)} s^{\frac{kq}{n}-1} \ds = O(1).
		\end{align*}

		Next, from \eqref{a}, we know $|\nabla^k u_\epsilon| = |\nabla^k \log(|x|)|$ in $B_1\setminus B_\epsilon$. Also for $1\leq k<n$, we know
		\begin{align*}
			|\nabla^k \log(|x|)|=\frac{\sqrt{\ell^k_n}}{|x|^{k}} \quad \text{ for } x\neq 0.
		\end{align*}
		Making use of this fact we deduce
		\begin{align*}
			\|\,|\nabla^k u_\epsilon|\,\|_{L^{\frac{n}{k},q}(B_1\setminus B_\epsilon)}^q=	\int_0^{\omega_n(1-\epsilon^n)} s^{\frac{kq}{n}-1} {\left| \chi_{B_1 \setminus B_\epsilon}(\cdot) \frac{\sqrt{\ell^k_n}}{|\cdot|^{k}}\right|^*}^q(s)\ds.
		\end{align*}
		For $s>0$, we know
		\begin{align*}
			\left|\chi_{B_1 \setminus B_\epsilon}(\cdot) \sqrt{\ell^k_n}\,|\cdot|^{-k}\right|^*(s) = \inf_{t\geq 0}\{t\,:\, \mathcal{L}^n( \{x\in B_1\setminus B_\epsilon: |\sqrt{\ell^k_n}\,|x|^{-k}|>t\})\leq s\}.
		\end{align*}
		Therefore,
		\begin{equation*}
			\left|\chi_{B_1 \setminus B_\epsilon}(\cdot) \sqrt{\ell^k_n}\,|\cdot|^{-k}\right|^*(s)=
			\left\{ \begin{array}{ll}
				0 & \text{ if } s\geq \omega_n(1-\epsilon^n); \\
				\frac{\sqrt{\ell^k_n}}{(s\omega_{n}^{-1}+\epsilon^n)^{\frac{k}{n}}} &\text{ if } s<\omega_n(1-\epsilon^n).
			\end{array} \right.
		\end{equation*}
		Hence, using this for small $\epsilon$, we have
		\begin{align*}
			&\|\,|\nabla^k u_\epsilon|\,\|_{L^{\frac{n}{k},q}(B_1\setminus B_\epsilon)}^q \\&=	\int_0^{\omega_n(1-\epsilon^n)} s^{\frac{kq}{n}-1} {|\nabla^k u_\epsilon|^*}^q(s)\ds\\&=\left(\sqrt{\ell^k_n}\right)^q\int_0^{\omega_n(1-\epsilon^n)}  \frac{s^{\frac{kq}{n}-1}}{(s\omega_{n}^{-1}+\epsilon^n)^{\frac{kq}{n}}} \ds\\&=\left(\omega_n^{\frac{k}{n}}\sqrt{\ell^k_n}\right)^q\int_0^{\epsilon^{-n}-1}  \frac{t^{\frac{kq}{n}-1}}{(t+1)^{\frac{kq}{n}}} \dt \\& \leq \left(\omega_n^{\frac{k}{n}}\sqrt{\ell^k_n}\right)^q\int_0^{1}  t^{\frac{kq}{n}-1} \dt + \left(\omega_n^{\frac{k}{n}}\sqrt{\ell^k_n}\right)^q\int_1^{\epsilon^{-n}} \frac{1}{t} \dt \\& = O(1) + \left( n^{\frac{1}{q}}\omega_n^{\frac{k}{n}}\sqrt{\ell^k_n}\right)^q \log \left(\frac{1}{\epsilon}\right).
		\end{align*}
		
		Therefore, with the help of Lemma~\ref{lor-break} and a combination of these three estimates, we deduce 
		\begin{align*}
			\|\,|\nabla^k u_\epsilon|\,\|_{L^{\frac{n}{k},q}(B_2)} &\leq  n^{\frac{1}{q}}\omega_n^{\frac{k}{n}}\sqrt{\ell^k_n} \left(\log \left(\frac{1}{\epsilon}\right)\right)^{1/q} + O(1)\\& \leq (1+\delta)^{\frac{1}{q^\prime}}\left( n^{\frac{1}{q}}\omega_n^{\frac{k}{n}}\sqrt{\ell^k_n}\right) \left(\log \left(\frac{1}{m\epsilon}\right)\right)^{1/q},
		\end{align*}
		holds for every $\delta>0$, $m \in \mathbb{N}$, and $\epsilon\leq \epsilon_0(\delta,m)$. Also, on $B_{m\epsilon}\setminus B_{\epsilon}$, when $m\epsilon \leq 1$  we have $u_\epsilon(x)=\log (1/|x|) \geq \log (1/m\epsilon)$. Furthermore, using \eqref{condition-1} we have, for $\epsilon>0$ sufficiently small, 
		\begin{align*}
			\nu(B_{m\epsilon}\setminus B_{\epsilon})\geq (C_0m^d-C_d')\epsilon^d=C(\nu) \epsilon^d.
		\end{align*}
		Hence, using this and by \eqref{eqn-1-th2.6} we deduce, for $\epsilon$ sufficiently small that
		\begin{align*}
			C&\geq \int_{B_2} \exp \left(  \frac{\beta|u_\epsilon(x)|}{\|\, |\nabla^k u_\epsilon|\, \|_{L^{\frac{n}{k},q}(B_2)}} \right)^{q^\prime} \dn \\& \geq 	\int_{B_{m\epsilon}\setminus B_{\epsilon}} \exp \left(  \frac{\beta|u_\epsilon(x)|}{(1+\delta)^{\frac{1}{q^\prime}}\left( n^{\frac{1}{q}}\omega_n^{\frac{k}{n}}\sqrt{\ell^k_n}\right) \left( \log \left(\frac{1}{m\epsilon}\right)\right)^{\frac{1}{q}}} \right)^{q^\prime} \dn \\& \geq C(\nu)\, \epsilon^d \exp \left(  \frac{\beta\log \left(\frac{1}{m\epsilon}\right)}{(1+\delta)^{\frac{1}{q^\prime}}\left( n^{\frac{1}{q}}\omega_n^{\frac{k}{n}}\sqrt{\ell^k_n}\right) \left( \log \left(\frac{1}{m\epsilon}\right)\right)^{\frac{1}{q}}} \right)^{q^\prime}\\& = \frac{C(\nu)}{2^d} \left(  \frac{1}{m\epsilon}\right)^{ \frac{\beta^{q^\prime}}{(1+\delta)\left( n^{\frac{1}{q}}\omega_n^{\frac{k}{n}}\sqrt{\ell^k_n}\right)^{q^\prime} }-d}.
		\end{align*}
		
For the above term to be finite as $\epsilon\to 0$, we must have
		\begin{align*}
			&\frac{\beta^{q^\prime}}{(1+\delta)\left( n^{\frac{1}{q}}\omega_n^{\frac{k}{n}}\sqrt{\ell^k_n}\right)^{q^\prime} }-d \leq 0 \\& \implies \beta \leq (1+\delta)^{\frac{1}{q^\prime}}d^{\frac{1}{q^\prime}} n^{\frac{1}{q}}\omega_n^{\frac{k}{n}}\sqrt{\ell^k_n}.
		\end{align*}
		Since $\delta$ is arbitrary, we then send $\delta \to 0$ and obtain $\beta \leq d^{\frac{1}{q^\prime}} n^{\frac{1}{q}}\omega_n^{\frac{k}{n}}\sqrt{\ell^k_n}$, which is the desired result.
	\end{proof}

	The last result of this section is the extension of Fontana-Morpurgo's result to the Lorentz scale/Alberico's result to the case of general measures, our Theorem~\ref{main-thm}.
	
	\begin{proof}[Proof of Theorem \ref{main-thm-1}]
		By \eqref{pt-wise}, making use of the definition of $\beta_{n,k,q}$ in \eqref{beta-n-k}, we have
\begin{align*}
			\beta_{n,k,q}\,|u(x)|	\leq \left(\frac{d}{n}\right)^{\frac{1}{q^\prime}}\gamma(k)\,\omega_n^{-\frac{n-k}{n}}I_k \left( |D^k u| \right) (x),
\end{align*}
		and the claimed inequality then follows as a consequence of Theorem~\ref{main-thm}.
		
		To prove the optimality of the constant in \eqref{eqn-adams-lorentz}, we argue as in the proof of Theorem \ref{main-thm-0}, mutatis mutandis.  Again, without loss of generality $\Omega =B_2$, while we utilize the same test functions.  Here the only change is that in place of the relation \eqref{natural} from Lemma \ref{log_constants}, we utilize the relations \eqref{adams} and \eqref{potential}, depending on parity.  Then the same argument invoking Lemma~\ref{lor-break} yields the bound 
		\begin{align*}
			\|\,|D^k u_\epsilon|\,\|_{L^{\frac{n}{k},q}(B_2)} &\leq  n^{\frac{1}{q}-1}\omega_n^{\frac{k}{n}-1}\gamma(k) \left(\log \left(\frac{1}{\epsilon}\right)\right)^{1/q} + O(1)\\& \leq (1+\delta)^{\frac{1}{q^\prime}}\left( n^{\frac{1}{q}-1}\omega_n^{\frac{k}{n}-1}\gamma(k)\right) \left(\log \left(\frac{1}{m\epsilon}\right)\right)^{1/q},
		\end{align*}
for $k$ even and
\begin{align*}
			\|\,|D^k u_\epsilon|\,\|_{L^{\frac{n}{k},q}(B_2)} &\leq  n^{\frac{1}{q}-1}\omega_n^{\frac{k}{n}-1}\tilde{\gamma}(k-1) \left(\log \left(\frac{1}{\epsilon}\right)\right)^{1/q} + O(1)\\& \leq (1+\delta)^{\frac{1}{q^\prime}}\left( n^{\frac{1}{q}-1}\omega_n^{\frac{k}{n}-1}\tilde{\gamma}(k-1)\right) \left(\log \left(\frac{1}{m\epsilon}\right)\right)^{1/q},
		\end{align*}
for $k$ odd.  These inequalities hold for every $\delta>0$, $m \in \mathbb{N}$, and $\epsilon\leq \epsilon_0(\delta,m)$. Again, on $B_{m\epsilon}\setminus B_{\epsilon}$, when $m\epsilon \leq 1$  we have $u_\epsilon(x)=\log (1/|x|) \geq \log (1/m\epsilon)$, and the measure $\nu$ satisfies a lower bound on this annulus.  

When $k$ is even this leads to the following computation analogous to that in the preceding theorem
\begin{align*}
			C&\geq \int_{B_2} \exp \left(  \frac{\beta|u_\epsilon(x)|}{\|\, |D^k u_\epsilon|\, \|_{L^{\frac{n}{k},q}(B_2)}} \right)^{q^\prime} \dn \\& \geq 	\int_{B_{m\epsilon}\setminus B_{\epsilon}} \exp \left(  \frac{\beta|u_\epsilon(x)|}{(1+\delta)^{\frac{1}{q^\prime}}\left( n^{\frac{1}{q}-1}\omega_n^{\frac{k}{n}-1}\gamma(k)\right) \left(\log \left(\frac{1}{m\epsilon}\right)\right)^{1/q}} \right)^{q^\prime} \dn \\& \geq C(\nu)\, \epsilon^d \exp \left(  \frac{\beta\log \left(\frac{1}{m\epsilon}\right)}{(1+\delta)^{\frac{1}{q^\prime}}\left( n^{\frac{1}{q}-1}\omega_n^{\frac{k}{n}-1}\gamma(k)\right) \left(\log \left(\frac{1}{m\epsilon}\right)\right)^{1/q}} \right)^{q^\prime}\\& = \frac{C(\nu)}{2^d} \left(  \frac{1}{m\epsilon}\right)^{ \frac{\beta^{q^\prime}}{(1+\delta)\left( n^{\frac{1}{q}-1}\omega_n^{\frac{k}{n}-1}\gamma(k)\right)^{q^\prime} }-d}.
		\end{align*}
		
Again, for the above term to be finite as $\epsilon\to 0$, we must have
\begin{align*}
			&\frac{\beta^{q^\prime}}{(1+\delta)\left( n^{\frac{1}{q}-1}\omega_n^{\frac{k}{n}-1}\gamma(k)\right)^{q^\prime} }-d \leq 0 \\& \implies \beta \leq (1+\delta)^{\frac{1}{q^\prime}}d^{\frac{1}{q^\prime}} n^{\frac{1}{q}-1}\omega_n^{\frac{k}{n}-1}\gamma(k).
		\end{align*}
Since $\delta$ is arbitrary, we next send $\delta \to 0$ and obtain $\beta \leq d^{\frac{1}{q^\prime}} n^{\frac{1}{q}-1}\omega_n^{\frac{k}{n}-1}\gamma(k) = \beta_{n,k,q}$, which is the desired result.  For $k$ odd, the inequality follows by a similar argument, we omit the details for brevity.  This completes the proof of the theorem.
\end{proof}

	\section{The Weak-Type Endpoint}\label{main_results_infty}
In this section we give the proofs of Theorems \ref{main-thm-infty}, \ref{main-thm-infty-1}, and  \ref{main-thm-infty-2}, addressing the higher order Adams-Cianchi inequalities in Lorentz spaces $L^{p,\infty}(\Omega)$ for $1<p<+\infty$ critical.
 
	\begin{proof}[Proof of Theorem \ref{main-thm-infty}]
		The assumption $||\, |f|\, ||_{L^{\frac{n}{\alpha},\infty}(\Omega)}\leq 1$ implies
		\begin{align*}
			f^*(s)\leq s^{-\frac{\alpha}{n}} \quad \quad \text{ for all }  s\in (0,|\Omega|).
		\end{align*}
Let $r$ be any admissible exponent in \eqref{eqn-oneil-lemma-exponent} and choose $\tau=t^{\frac{n}{d}}$ in \eqref{eqn-oneil-lemma-rearrangement}. Then using the above we obtain
		\begin{align*}
			( I_\alpha f)^*(t)&\leq (I_\alpha f)^{**}(t)\\&\leq Ct^{-\frac{1}{\theta}}\int_0^{t} f^*(s^{\frac{n}{d}})s^{-1+\frac{n}{dr}} ds+\frac{n}{d\gamma(\alpha)}\omega_n^{\frac{n-\alpha}{n}}\int_{t}^{|\Omega|^{\frac{d}{n}}}  f^*(s^{\frac{n}{d}})  s^{-1+\frac{\alpha}{d}} ds\\&\leq Ct^{-\frac{1}{\theta}}\int_0^{t} s^{-1+\frac{n}{dr}-\frac{\alpha}{d}} ds+\frac{n}{d\gamma(\alpha)}\omega_n^{\frac{n-\alpha}{n}}\int_{t}^{|\Omega|^{\frac{d}{n}}}   s^{-1} ds\\&= A + \frac{n}{d\gamma(\alpha)}\omega_n^{\frac{n-\alpha}{n}}\log\left(\frac{1}{t}\right),
		\end{align*}
		for every $t>0$.	Hence, taking  $t_1=\max\{\nu(\Omega), |\Omega|^{\frac{d}{n}}\}$ and for $\gamma<\left(\frac{d}{n}\right)\gamma(\alpha)\,\omega_n^{-\frac{n-\alpha}{n}}$, we find
		\begin{align*}
			\int_\Omega \exp \left(\gamma |I_\alpha f(x)| \right) \dn &	= 	\int_0^{\nu(\Omega)} \exp \left( \gamma(I_\alpha f)^*(t) \right) \dt \\&\leq \int_{0}^{t_1}\exp \left( \gamma A + \frac{n \gamma}{d\gamma(\alpha)}\omega_n^{\frac{n-\alpha}{n}}\log\left(\frac{1}{t}\right) \right) \dt <+\infty.
		\end{align*}

To verify the optimality, the failure of the estimate at the endpoint, we make use of a variant of the test functions used in the proof of Theorem~\ref{main-thm}.  Without loss of generality, we let $\Omega=B$, the unit ball centered at the origin and define
\begin{align*}
f(x) := \frac{\chi_B(x)}{n\omega_n} \frac{1}{|x|^\alpha}.
\end{align*} 
By Lemma \ref{weak-type_norm} one has the bound
\begin{align*}
\|f\|_{L^{n/\alpha,\infty}(B)} \leq \frac{1}{n} \omega_n^{\alpha/n-1}.
\end{align*}
Let $B_r$ be the ball of radius $r$ with $0<r<1$ and centered at the origin.  By the computation in Theorem \ref{main-thm}, one has
\begin{align*}
\gamma(\alpha) I_\alpha f(x) = \log \frac{1}{|x|} +O(1)
\end{align*}
for $|x| \leq r$, so that
\begin{align*}
\gamma(\alpha) |I_\alpha f(x)| \geq \log \frac{1}{|x|} -C
\end{align*}
for $|x| \leq r$, $r>0$ sufficiently small.

When $\beta= \frac{d}{n} \gamma(\alpha) \omega_n^{- \frac{n-\alpha}{n}}$, for this $f$ we have
\begin{align}\label{eqn-th1.5-1}
				\int_B \exp \left(\frac{\beta|I_\alpha f(x)|}{\|\, |f|\, \|_{L^{\frac{n}{\alpha},\infty}(B)}} \right) \dn \geq \int_{B_r} \exp \left( \frac{d}{n} \omega_n^{- \frac{n-\alpha}{n}} \frac{\log \frac{1}{|x|} - C}{\frac{1}{n} \omega_n^{\alpha/n-1}} \right) \dn.
			\end{align}
In particular, for $r>0$ sufficiently small one has
\begin{align*}
				\int_B \exp \left(\frac{\beta|I_\alpha f(x)|}{\|\, |f|\, \|_{L^{\frac{n}{\alpha},\infty}(B)}} \right) \dn &\geq \exp \left(-Cd \right) \int_{B_r}  \frac{1}{|x|^d} \dn\\
				&= \exp\left(-Cd \right)\int_{1/r^d}^\infty \nu\left( \left\{ \frac{1}{t^{1/d}}> |x|\right\}\right)\;dt \\
			& \geq C_0\exp\left(-Cd \right)	\int_{1/r^d}^\infty \frac{1}{t}\;dt\\
			&=+\infty.
\end{align*}
This shows the failure of the estimate at the endpoint.
\end{proof}
	
	We next give the
	\begin{proof}[Proof of Theorem \ref{main-thm-infty-1}]
		Let $\gamma< d \,\omega_n^{\frac{k}{n}}\,\sqrt{\ell^k_n}$.  From \eqref{pt-wise-org}, we deduce
		\begin{align*}
			\gamma|u(x)|	\leq \gamma \frac{\gamma(k)}{n\omega_{n}\sqrt{\ell^k_n}}\,I_k \left( |\nabla^k u| \right) (x).
		\end{align*}
The assumption on  $\gamma$ implies $\gamma \frac{\gamma(k)}{n\omega_{n}\sqrt{\ell^k_n}}<\left(\frac{d}{n}\right)\gamma(k)\,\omega_n^{-\frac{n-k}{n}}$.  Therefore, the desired result follows from Theorem~\ref{main-thm-infty}.  To see that the constant is sharp, we consider the test functions used in \cite[p.~398]{Alberico}.  Without loss of generality, we take $\Omega=B$, the unit ball, and we define
		\begin{align*}
			u_a(x)=
			\left\{ \begin{array}{ll}
				\varphi\left(\frac{\log \frac{1}{|x|}}{\log a}\right)  & \text{ if } x \in B; \\
				0 & \text{ otherwise, }
			\end{array} \right.
		\end{align*}
		for any increasing smooth function $\varphi: \mathbb{R} \to [0,\infty)$  
		that satisfies $|\varphi'(t)|\leq 1$ and
		\begin{align*}
			\varphi(t)=
			\left\{ \begin{array}{ll}
				0  & \text{ if } t\leq 0; \\
				t -1/2& \text{ if } t\geq 1.
			\end{array} \right.
		\end{align*}
		For this test function, by the product and chain rule, one has that
		\begin{align*}
			|\nabla^k u_a(x)| \leq \frac{1}{\log a} \frac{\sqrt{\ell^k_n}}{|x|^k}\left(1+ \frac{C}{(\log a)^2}\right).
		\end{align*}
		Therefore, by Lemma \ref{weak-type_norm} we have
		\begin{align*}
			\|\nabla^k u_a\|_{L^{n/k,\infty}(B)} \leq \frac{ \omega_n^{k/n}\sqrt{\ell^k_n} }{\log a} \left(1+ \frac{C}{(\log a)^2}\right).
		\end{align*}
		With this inequality, we find
		\begin{align*}
			\int_{B}&\exp\left(\frac{\gamma |u_{a}(x)|}{\|\nabla^k u_a\|_{L^{n/k,\infty}(B)}}\right)\dn \\
			&\;\;\geq \int_{B_{1/a}} \exp\left(\gamma \frac{(\frac{\log \frac{1}{|x|}}{\log a} - 1/2) \log a }{ \omega_n^{k/n} \sqrt{\ell^k_n} \left(1+ \frac{C}{(\log a)^2}\right)}\right)\dn\\
			&=\exp\left(\gamma \frac{ -  \log a }{ 2\omega_n^{k/n} \sqrt{\ell^k_n} \left(1+ \frac{C}{(\log a)^2}\right)}\right) \int_{B_{1/a}}  \frac{1}{|x|^{\gamma_a}}  \dn
		\end{align*}		
		for
		\begin{align*}
			\gamma_a:= \frac{\gamma}{ \omega_n^{k/n} \sqrt{\ell^k_n} \left(1+ \frac{C}{(\log a)^2}\right)}.
		\end{align*}
		However, for $a$ large enough that \eqref{condition-1} is satisfied and $d>\gamma_a$, we have
		\begin{align*}
			\int_{B_{1/a}}  \frac{1}{|x|^{\gamma_a}}  \dn &= \int_{a^{\gamma_a}}^\infty \nu( \{ \frac{1}{t^{1/{\gamma_a}}} >|x|\})\;dt \\
			&\geq  \int_{a^{\gamma_a}}^\infty C_0 \frac{1}{t^{d/{\gamma_a}}}\;dt\\
			&= +\infty.
		\end{align*}

	\end{proof}

	\begin{proof}[Proof of Theorem \ref{main-thm-infty-2}]
		Let $\gamma<\beta_{n,k,\infty}$.  Using \eqref{beta-n-k} and \eqref{pt-wise}, we deduce
		\begin{align*}
			\gamma|u(x)|	\leq \tilde{\gamma}_k I_k \left( |D^k u| \right) (x),
		\end{align*}
		where 
		\begin{equation*}
			\tilde{\gamma}_k=
			\left\{ \begin{array}{ll}
				\gamma & \text{ if } k \text { is even}; \\
				\gamma \frac{\gamma(k)}{\widetilde{\gamma}(k-1)} &\text{ if }  k \text { is odd}.
			\end{array} \right.
		\end{equation*}
The assumption on $\gamma$ implies $\tilde{\gamma}_k <\beta_{n,k,\infty}$, and therefore Theorem~\ref{main-thm-infty} yields the desired inequality.
		
Concerning optimality, we utilize the test functions used in the proof of Theorem \ref{main-thm-infty-1} and argue the case of $k$ odd, as the computation for $k$ even is just a replacement of constants.  By the product and chain rule, one has that
		\begin{align*}
			|D^k u_a(x)| \leq \frac{1}{\log a} \frac{\tilde{\gamma}(k-1)}{n \omega_n} \frac{1}{|x|^k}\left(1+ \frac{C}{(\log a)^2}\right).
		\end{align*}
		Therefore, by Lemma \ref{weak-type_norm} we have
		\begin{align*}
			\|\nabla^k u_a\|_{L^{n/k,\infty}(B)} \leq  \frac{\omega_n^{k/n-1}\tilde{\gamma}(k-1)}{n } \frac{1 }{\log a} \left(1+ \frac{C}{(\log a)^2}\right).
		\end{align*}
		
As in the proof of the preceding Theorem's optimality, we find 
		\begin{align*}
			\int_{B}&\exp\left(\frac{\gamma |u_{a}(x)|}{\|D^k u_a\|_{L^{n/k,\infty}(B)}}\right)\dn \\
			&\;\;\geq \int_{B_{1/a}} \exp\left(\gamma \frac{(\frac{\log \frac{1}{|x|}}{\log a} - 1/2) \log a }{\frac{\omega_n^{k/n-1}\tilde{\gamma}(k-1)}{n } \frac{1 }{\log a} \left(1+ \frac{C}{(\log a)^2}\right) }\right)\dn\\
			&=\exp\left(\gamma \frac{ -  \log a }{ 2\frac{\omega_n^{k/n-1}\tilde{\gamma}(k-1)}{n } \left(1+ \frac{C}{(\log a)^2}\right)}\right) \int_{B_{1/a}}  \frac{1}{|x|^{\gamma_a}}  \dn
		\end{align*}		
		for
		\begin{align*}
			\gamma_a:= \frac{\gamma}{\frac{\omega_n^{k/n-1}\tilde{\gamma}(k-1)}{n } \left(1+ \frac{C}{(\log a)^2}\right)}.
		\end{align*}
		However, for $a$ large enough that \eqref{condition-1} is satisfied and $d>\gamma_a$, we have
		\begin{align*}
			\int_{B_{1/a}}  \frac{1}{|x|^{\gamma_a}}  \dn &= \int_{a^{\gamma_a}}^\infty \nu( \{ \frac{1}{t^{1/{\gamma_a}}} >|x|\})\;dt \\
			&\geq  \int_{a^{\gamma_a}}^\infty C_0 \frac{1}{t^{d/{\gamma_a}}}\;dt\\
			&= +\infty.
		\end{align*}

	\end{proof}
	
\section{Trace Hansson-Brezis-Wainger}\label{HBW}
We begin this section with the
\begin{proof}[Proof of Theorem \ref{traceHBW}]
As in the proofs of the preceding theorems, we let $r$ be any acceptable exponent within the range specified in \eqref{eqn-oneil-lemma-exponent} and choose $\tau=t^{\frac{n}{d}}$ in \eqref{eqn-oneil-lemma-rearrangement} to obtain
\begin{align*}
			(I_\alpha f)^{*}(t) \leq (I_\alpha f)^{**}(t)\leq Ct^{-\frac{1}{\theta}}\int_0^{t^{\frac{n}{d}}} f^*(u)u^{-1+\frac{1}{r}} \du+\frac{1}{\gamma(\alpha)}\omega_n^{\frac{n-\alpha}{n}}\int_{t^{\frac{n}{d}}}^{|\Omega|}  f^*(u)  u^{-\frac{n-\alpha}{n}} \du.
		\end{align*}

This pointwise estimate and the triangle inequality imply 
\begin{align}\label{trBW1}
\left(\int_{0}^{\nu(\Omega)}\frac{\left((I_\alpha f)^{*}(t)\right)^q}{\left(1+|\log \frac{t}{\nu(\Omega)}|\right)^q}\frac{\dt}{t}\right)^{1/q} \leq  \I_1 + \I_2,
\end{align}
for 
		\begin{align*}
			\I_1^q=\int_{0}^{\nu(\Omega)}\frac{\left(\int_0^{t^{\frac{n}{d}}} f^*(u)u^{-1+\frac{1}{r}} \du\right)^q}{\left(1+|\log \frac{t}{\nu(\Omega)}|\right)^q}\frac{\dt}{t^{1+\frac{q}{\theta}}},
		\end{align*}
		and 
		\begin{align*}
			\I_2^q=\int_{0}^{\nu(\Omega)}\frac{\left(\int_{t^{\frac{n}{d}}}^{|\Omega|}  f^*(u)  u^{-\frac{n-\alpha}{n}} \du\right)^q}{\left(1+|\log \frac{t}{\nu(\Omega)}|\right)^q}\frac{\dt}{t}.
		\end{align*}
We next estimate $\I_1$ and $\I_2$ via Hardy inequalities.  Concerning $\I_1$, by the change of variables $s=t^{\frac{n}{d}}$ we have
		\begin{align*}
			\I_1^q&\leq \int_{0}^{\nu(\Omega)}t^{-{1-\frac{q}{\theta}}}\left(\int_0^{t^{\frac{n}{d}}} f^*(u)u^{-1+\frac{1}{r}} \du\right)^q\dt\\&\leq \int_{0}^{\infty}t^{-{1-\frac{q}{\theta}}}\left(\int_0^{t^{\frac{n}{d}}} f^*(u)u^{-1+\frac{1}{r}} \du\right)^q\dt\\&=\frac{d}{n}\int_{0}^{\infty}s^{-{1-\frac{dq}{n\theta}}}\left(\int_0^{s} f^*(u)u^{-1+\frac{1}{r}} \du\right)^q\ds.
		\end{align*}
An application of the Hardy inequality recorded above in Lemma \ref{another-hardy} with $p=q \in (1,\infty)$, $w={1+\frac{dq}{n\theta}}>1$, and $\psi(u)=f^*(u)u^{-1+\frac{1}{r}}$, yields
		\begin{align*}
			\I_1^q\lesssim \int_{0}^{\infty}s^{-\frac{dq}{n\theta}+\frac{q}{r}} {f^*}^q(s) \frac{\ds}{s}=\int_{0}^{\infty}\left(s^{\frac{n\theta-d}{nr\theta}} {f^*}(s)\right)^q \frac{\ds}{s}=\|f\|^q_{L^{\frac{nr\theta}{n\theta-d},q}(\Omega)}.
		\end{align*}
The relations $\theta=\frac{rd}{n-\alpha r}$, $r<\frac{n}{\alpha}$, and $1<r$, imply
\begin{align*}
			\frac{n\theta-d}{nr\theta}-\frac{\alpha}{n}=\frac{1}{r}-\frac{d}{nr\theta}-\frac{\alpha}{n}=\left(\frac{1}{r}-\frac{\alpha}{n}\right) \left(1-\frac{1}{r}\right)>0.
\end{align*}
In particular, $\frac{nr\theta}{n\theta-d}<\frac{n}{\alpha}$, and since $|\Omega|<+\infty$, using the nested embedding with respect to the second variable of Lorentz space, we have
\begin{align}\label{trBW2}
			\I_1^q\lesssim\|f\|^q_{L^{\frac{n}{\alpha},q}(\Omega)}.
		\end{align}
		
To estimate $\I_2$, we define  $t_1=\max\{\nu(\Omega), |\Omega|^{\frac{d}{n}}\}$ and make the change of variables $s^{\frac{n}{d}}=u$ to obtain
		\begin{align*}
			\I_2^q&=\int_{0}^{\nu(\Omega)}\frac{\left(\int_{t^{\frac{n}{d}}}^{|\Omega|}  f^*(u)  u^{-\frac{n-\alpha}{n}} \du\right)^q}{\left(1+|\log \frac{t}{\nu(\Omega)}|\right)^q}\frac{\dt}{t}\\&=\left(\frac{n}{d}\right)^q\int_{0}^{\nu(\Omega)}\frac{\left(\int_{t}^{|\Omega|^{\frac{d}{n}}}  f^*(s^{\frac{n}{d}})  s^{-\frac{d-\alpha}{d}} \ds\right)^q}{\left(1+|\log \frac{t}{\nu(\Omega)}|\right)^q}\frac{\dt}{t}\\&\leq \left(\frac{n}{d}\right)^q \int_{0}^{t_1}\frac{\left(\int_{t}^{t_1}  f^*(s^{\frac{n}{d}})  s^{-\frac{d-\alpha}{d}} \ds\right)^q}{\left(1+|\log \frac{t}{t_1}|\right)^q}\frac{\dt}{t}.
		\end{align*}
Here an application of the logarithmic Hardy inequality  recorded above in Lemma~\ref{log-hardy} with $p=q$, $R=t_1$, and $\psi(s)=f^*(s^{\frac{n}{d}})  s^{-\frac{d-\alpha}{d}}$ yields
		\begin{align*}
			\I_2^q \lesssim \int_{0}^{t_1}t^{q-1}{f^*}^q(t^{\frac{n}{d}})  t^{-\frac{(d-\alpha)q}{d}}\dt.
		\end{align*}
Finally, the change of variable $t^{\frac{n}{d}}=s$ and the observation that $f^*(s)=0$ for $s\geq |\Omega|$ yields
		\begin{align}\label{trBW3}
			\I_2^q \lesssim \int_{0}^{\infty}\left({f^*}(s)  s^{\frac{\alpha }{n}}\right)^q\frac{\ds}{s}=\|f\|^q_{L^{\frac{n}{\alpha},q}(\Omega)}.
		\end{align}
The combination of \eqref{trBW1}, \eqref{trBW2}, and \eqref{trBW3} yields the claimed inequality and thus completes the proof.
	\end{proof}

We conclude the paper with the 

\begin{proof}[Proof of Corollary \ref{diff_traceHBW}]
This follows immediately from Theorem \ref{traceHBW} and the inequality \eqref{representation_inequality}.
\end{proof}

	\section*{Acknowledgments} 
	The author P.~Roychowdhury is supported by the National Theoretical Science Research Center Operational Plan (Project number: 112L104040).  D.~Spector is supported by the National Science and Technology Council of Taiwan under research grant numbers 110-2115-M-003-020-MY3/113-2115-M-003-017-MY3 and the Taiwan Ministry of Education under the Yushan Fellow Program.


\begin{bibdiv}
		\begin{biblist}
			
		
			\bib{Adams1988}{article}{
				author={Adams, David R.},
				title={A sharp inequality of J. Moser for higher order derivatives},
				journal={Ann. of Math. (2)},
				volume={128},
				date={1988},
				number={2},
				pages={385--398},
				issn={0003-486X},
				review={\MR{960950}},
				doi={10.2307/1971445},
			}
			
			\bib{Adams1973}{article}{
				author={Adams, David R.},
				title={Traces of potentials. II},
				journal={Indiana Univ. Math. J.},
				volume={22},
				date={1972/73},
				pages={907--918},
				issn={0022-2518},
				review={\MR{313783}},
				doi={10.1512/iumj.1973.22.22075},
			}
			
						
			\bib{AdamsXiao}{article}{
				author={Adams, David R.},
				author={Xiao, Jie},
				title={Morrey potentials and harmonic maps},
				journal={Comm. Math. Phys.},
				volume={308},
				date={2011},
				number={2},
				pages={439--456},
				issn={0010-3616},
				review={\MR{2851148}},
				doi={10.1007/s00220-011-1319-5},
			}
			
			
			\bib{AdamsXiao2}{article}{
				author={Adams, David R.},
				author={Xiao, Jie},
				title={Erratum to: Morrey potentials and harmonic maps [MR2851148]},
				journal={Comm. Math. Phys.},
				volume={339},
				date={2015},
				number={2},
				pages={769--771},
				issn={0010-3616},
				review={\MR{3370618}},
				doi={10.1007/s00220-015-2409-6},
			}
			
			\bib{Alberico}{article}{
				author={Alberico, Angela},
				title={Moser type inequalities for higher-order derivatives in Lorentz
					spaces},
				journal={Potential Anal.},
				volume={28},
				date={2008},
				number={4},
				pages={389--400},
				issn={0926-2601},
				review={\MR{2403289}},
				doi={10.1007/s11118-008-9085-5},
			}
			
			
			\bib{AFT}{article}{
				author={Alvino, Angelo},
				author={Ferone, Vincenzo},
				author={Trombetti, Guido},
				title={Moser-type inequalities in Lorentz spaces},
				journal={Potential Anal.},
				volume={5},
				date={1996},
				number={3},
				pages={273--299},
				issn={0926-2601},
				review={\MR{1389498}},
				doi={10.1007/BF00282364},
			}
			
			\bib{BW}{article}{
				author={Br\'ezis, Ha\"im},
				author={Wainger, Stephen},
				title={A note on limiting cases of Sobolev embeddings and convolution
					inequalities},
				journal={Comm. Partial Differential Equations},
				volume={5},
				date={1980},
				number={7},
				pages={773--789},
				issn={0360-5302},
				review={\MR{0579997}},
				doi={10.1080/03605308008820154},
			} 

			\bib{Chen-Spector}{article}{
   author={Chen, You-Wei},
   author={Spector, Daniel},
   title={On functions of bounded $\beta$-dimensional mean oscillation},
   journal={Adv. Calc. Var.},
   volume={17},
   date={2024},
   number={3},
   pages={975--996},
   issn={1864-8258},
   review={\MR{4767358}},
   doi={10.1515/acv-2022-0084},
}
		
			\bib{Cianchi:2008}{article}{
				author={Cianchi, Andrea},
				title={Moser-Trudinger trace inequalities},
				journal={Adv. Math.},
				volume={217},
				date={2008},
				number={5},
				pages={2005--2044},
				issn={0001-8708},
				review={\MR{2388084}},
				doi={10.1016/j.aim.2007.09.007},
			}
			
			\bib{CP}{article}{
				author={Cwikel, Michael},
				author={Pustylnik, Evgeniy},
				title={Sobolev type embeddings in the limiting case},
				journal={J. Fourier Anal. Appl.},
				volume={4},
				date={1998},
				number={4-5},
				pages={433--446},
				issn={1069-5869},
				review={\MR{1658620}},
				doi={10.1007/BF02498218},
			}
			
			
			\bib{FontanaMorpurgo}{article}{
				author={Fontana, Luigi},
				author={Morpurgo, Carlo},
				title={Adams inequalities on measure spaces},
				journal={Adv. Math.},
				volume={226},
				date={2011},
				number={6},
				pages={5066--5119},
				issn={0001-8708},
				review={\MR{2775895}},
				doi={10.1016/j.aim.2011.01.003},
			}
			
			
			\bib{FontanaMorpurgo:na}{article}{
				author={Fontana, Luigi},
				author={Morpurgo, Carlo},
				title={Adams inequalities for Riesz subcritical potentials},
				journal={Nonlinear Anal.},
				volume={192},
				date={2020},
				pages={111662, 32},
				issn={0362-546X},
				review={\MR{4021978}},
				doi={10.1016/j.na.2019.111662},
			}

			
			
			
						
		\bib{Fuglede}{article}{
				author={Fuglede, Bent},
				title={The logarithmic potential in higher dimensions},
				journal={Mat.-Fys. Medd. Danske Vid. Selsk.},
				volume={33},
				date={1960},
				number={1},
				pages={14 (1960)},
				issn={0023-3323},
				review={\MR{0125248}},
			}
			
			\bib{GS}{article}{
				author={Garg, Rahul},
				author={Spector, Daniel},
				title={On the regularity of solutions to Poisson's equation},
				journal={C. R. Math. Acad. Sci. Paris},
				volume={353},
				date={2015},
				number={9},
				pages={819--823},
				issn={1631-073X},
				review={\MR{3377679}},
				doi={10.1016/j.crma.2015.07.001},
			}
			
			\bib{GS1}{article}{
				author={Garg, Rahul},
				author={Spector, Daniel},
				title={On the role of Riesz potentials in Poisson's equation and Sobolev
					embeddings},
				journal={Indiana Univ. Math. J.},
				volume={64},
				date={2015},
				number={6},
				pages={1697--1719},
				issn={0022-2518},
				review={\MR{3436232}},
				doi={10.1512/iumj.2015.64.5706},
			}
			
			\bib{grafakos}{book}{
				author={Grafakos, Loukas},
				title={Classical Fourier analysis},
				series={Graduate Texts in Mathematics},
				volume={249},
				edition={3},
				publisher={Springer, New York},
				date={2014},
				pages={xviii+638},
				isbn={978-1-4939-1193-6},
				isbn={978-1-4939-1194-3},
				review={\MR{3243734}},
				doi={10.1007/978-1-4939-1194-3},
			}
			
			\bib{Hansson}{article}{
   author={Hansson, Kurt},
   title={Imbedding theorems of Sobolev type in potential theory},
   journal={Math. Scand.},
   volume={45},
   date={1979},
   number={1},
   pages={77--102},
   issn={0025-5521},
   review={\MR{0567435}},
   doi={10.7146/math.scand.a-11827},
}

		
		
			
			
			\bib{H:1972}{article}{
				author={Hedberg, Lars Inge},
				title={On certain convolution inequalities},
				journal={Proc. Amer. Math. Soc.},
				volume={36},
				date={1972},
				pages={505--510},
				issn={0002-9939},
				review={\MR{312232}},
				doi={10.2307/2039187},
			}
			\bib{Y:1962}{article}{
				author={Judovi\v{c}, V. I.},
				title={Some bounds for solutions of elliptic equations},
				language={Russian},
				journal={Mat. Sb. (N.S.)},
				volume={59 (101)},
				date={1962},
				number={suppl.},
				pages={229--244},
				review={\MR{0149062}},
			}
			
			
			\bib{MS}{article}{
				author={Mart\'{\i}nez, \'{A}ngel D.},
				author={Spector, Daniel},
				title={An improvement to the John-Nirenberg inequality for functions in
					critical Sobolev spaces},
				journal={Adv. Nonlinear Anal.},
				volume={10},
				date={2021},
				number={1},
				pages={877--894},
				issn={2191-9496},
				review={\MR{4191703}},
				doi={10.1515/anona-2020-0157},
			}
			
			\bib{Mazyabook}{book}{
   author={Maz'ya, Vladimir},
   title={Sobolev spaces with applications to elliptic partial differential
   equations},
   series={Grundlehren der mathematischen Wissenschaften [Fundamental
   Principles of Mathematical Sciences]},
   volume={342},
   edition={augmented edition},
   publisher={Springer, Heidelberg},
   date={2011},
   pages={xxviii+866},
   isbn={978-3-642-15563-5},
   review={\MR{2777530}},
   doi={10.1007/978-3-642-15564-2},
}
			
						\bib{MoriiSatoSawano}{article}{
				author={Morii, Kei},
				author={Sato, Tokushi},
				author={Sawano, Yoshihiro},
				title={Certain identities on derivatives of radial homogeneous and
					logarithmic functions},
				journal={Commun. Math. Anal.},
				volume={9},
				date={2010},
				number={2},
				pages={51--66},
				issn={1938-9787},
				review={\MR{2737754}},
			}
			


			


			\bib{Moser}{article}{
				author={Moser, J.},
				title={A sharp form of an inequality by N. Trudinger},
				journal={Indiana Univ. Math. J.},
				volume={20},
				date={1970/71},
				pages={1077--1092},
				issn={0022-2518},
				review={\MR{301504}},
				doi={10.1512/iumj.1971.20.20101},
			}
			
			\bib{oneil}{article}{
				author={O'Neil, Richard},
				title={Convolution operators and $L(p,\,q)$ spaces},
				journal={Duke Math. J.},
				volume={30},
				date={1963},
				pages={129--142},
				issn={0012-7094},
				review={\MR{0146673}},
			}
			
			
			\bib{P:1966}{article}{
				author={Peetre, Jaak},
				title={Espaces d'interpolation et th\'{e}or\`eme de Soboleff},
				language={French},
				journal={Ann. Inst. Fourier (Grenoble)},
				volume={16},
				date={1966},
				number={fasc. 1},
				pages={279--317},
				issn={0373-0956},
				review={\MR{221282}},
			}
			\bib{P:1965}{article}{
				author={Poho\v{z}aev, S. I.},
				title={On the eigenfunctions of the equation $\Delta u+\lambda f(u)=0$},
				language={Russian},
				journal={Dokl. Akad. Nauk SSSR},
				volume={165},
				date={1965},
				pages={36--39},
				issn={0002-3264},
				review={\MR{0192184}},
			}
			
			\bib{RSS}{article}{
   author={Rai\c t\u a, Bogdan},
   author={Spector, Daniel},
   author={Stolyarov, Dmitriy},
   title={A trace inequality for solenoidal charges},
   journal={Potential Anal.},
   volume={59},
   date={2023},
   number={4},
   pages={2093--2104},
   issn={0926-2601},
   review={\MR{4684387}},
   doi={10.1007/s11118-022-10008-x},
}
			
			\bib{Shafrir}{article}{
				author={Shafrir, Itai},
				title={The best constant in the embedding of $W^{N,1}(\Bbb R^N)$ into
					$L^\infty(\Bbb R^N)$},
				journal={Potential Anal.},
				volume={50},
				date={2019},
				number={4},
				pages={581--590},
				issn={0926-2601},
				review={\MR{3938494}},
				doi={10.1007/s11118-018-9695-5},
			}
			
			\bib{SS}{article}{
				author={Shafrir, Itai},
				author={Spector, Daniel},
				title={Best constants for two families of higher order critical Sobolev
					embeddings},
				journal={Nonlinear Anal.},
				volume={177},
				date={2018},
				number={part B},
				part={part B},
				pages={753--769},
				issn={0362-546X},
				review={\MR{3886600}},
				doi={10.1016/j.na.2018.04.027},
			}
			
				
			\bib{S}{book}{
				author={Stein, Elias M.},
				title={Singular integrals and differentiability properties of functions},
				series={Princeton Mathematical Series, No. 30},
				publisher={Princeton University Press, Princeton, N.J.},
				date={1970},
				pages={xiv+290},
				review={\MR{0290095}},
			}
			
			
			\bib{S:1971}{article}{
				author={Strichartz, Robert S.},
				title={A note on Trudinger's extension of Sobolev's inequalities},
				journal={Indiana Univ. Math. J.},
				volume={21},
				date={1971/72},
				pages={841--842},
				issn={0022-2518},
				review={\MR{293389}},
				doi={10.1512/iumj.1972.21.21066},
			}
			\bib{T:1967}{article}{
				author={Trudinger, Neil S.},
				title={On imbeddings into Orlicz spaces and some applications},
				journal={J. Math. Mech.},
				volume={17},
				date={1967},
				pages={473--483},
				review={\MR{0216286}},
				doi={10.1512/iumj.1968.17.17028},
			}
			

			
%
			
%
%
%


			

			
		\end{biblist}
		
	\end{bibdiv}
\end{document}